\documentclass[a4paper,11pt]{amsart}

\usepackage[hidelinks]{hyperref}

\usepackage{amsmath,amsthm,amssymb,latexsym,amsfonts,wrapfig}
\usepackage[latin1]{inputenc}
\usepackage[english,activeacute]{babel}
\usepackage{graphicx,xcolor}
\usepackage{enumerate}
\theoremstyle{definition}  
\usepackage{parskip}
\usepackage{hyperref}
\usepackage{cleveref}

\usepackage{todonotes}

\newcommand{\norm}[1]{\left\lVert#1\right\rVert}
\allowdisplaybreaks

\def \N{\mathbb N}

\def \R{\mathbb R}
\def \C{\mathbb C}

\def \pa{{\partial}}

\def \O{\mathcal{O}}

\def \E{\mathcal{E}}

\numberwithin{equation}{section}

\theoremstyle{plain}
\newtheorem{thm}{Theorem}[section]
\newtheorem{lem}[thm]{Lemma}
\newtheorem{prop}[thm]{Proposition}
\newtheorem{cor}[thm]{Corollary}
\theoremstyle{definition}

\theoremstyle{remark}
\newtheorem{rk}[thm]{Remark}

\theoremstyle{plain}

\theoremstyle{remark}

\theoremstyle{plain}

\title{On the nonlinear Dysthe equation}

\author[R. Grande]{Ricardo Grande$^1$}
\address{Department of Mathematics, Massachusetts Institute of Technology, Cambridge, MA}
\email{rgi@mit.edu} 

\thanks{$^1$  R.G. was funded in part by  the Simons Foundation.}

\author[K. Kurianski,]{Kristin M. Kurianski$^2$}
\address{Department of Mathematics, Massachusetts Institute of Technology, Cambridge, MA}
\email{kmdett@mit.edu}

\thanks{$^2$  K.M.K.   was funded in part  by the National Science Foundation Graduate Research Fellowship under Grant No. 1122374.}

\author[G. Staffilani]{Gigliola Staffilani$^3$}
\address{Department of Mathematics, Massachusetts Institute of Technology, Cambridge, MA}
\email{gigliola@mit.edu} 

\thanks{$^3$  G.S. was funded in part by NSF  DMS-1764403, and the Simons Foundation. }

\begin{document}
	
	\begin{abstract}
	This work is dedicated to putting on a solid analytic ground the theory of local well-posedness for the two dimensional
	Dysthe equation. This equation  can be derived from the incompressible Navier-Stokes equation after performing an asymptotic expansion of a wavetrain modulation to the fourth order. Recently, this equation has been used to numerically study rare phenomena on large water bodies such as rogue waves. In order to study well-posedness, we use Strichartz, and improved smoothing and maximal function estimates. We follow ideas from the pioneering work of Kenig, Ponce and Vega, but since the equation is highly anisotropic,   several technical challenges had to be resolved.  We conclude our work by also presenting an ill-posedness result.	
		\end{abstract}

	\maketitle
	
\section{Introduction}
	
	\subsection{Background}
	
	Ocean waves are called rogue or freak waves when their amplitude exceeds twice the characteristic wave height expected for the given surface conditions \cite{roguereview}. Such unexpected extreme events pose a threat of catastrophic impacts for a variety of naval infrastructure and, therefore, are important to understand. Over the past several decades, there have been many efforts to model and predict the behavior of rogue waves (e.g., \cite{CouSaps,Dysthe,roguereview,FarSap,Hasselmann}).
	
	The motivation for this paper came from the work of Farazmand and Sapsis \cite{FarSap} who studied numerical simulations of large wave prediction for two-dimensional water waves. Their work supports the hypothesis that large ocean waves can be caused by nonlinear interactions as a result of focusing, although the precise mechanism for the formation of such extreme events is a subject of much debate (see for instance \cite{SapsisMechanism,OnoratoMechanism2,OnoratoMechanism} and the references therein). 
	By decomposing the surface wavefield into localized Gaussian wave groups and evolving the groups according to the governing envelope equations, Farazmand and Sapsis computed the expected maxima of each group and produced a prediction of the maximal future amplitude generated by given initial data. The governing envelope equation they used is given by the two-dimensional Dysthe equation:
		\begin{equation}\label{eq:Dysthe}
	\left\lbrace\begin{array}{ll}
	\pa_t u + L(u)  = N(u),\quad t\in\R,\ (x,y)\in\R^2\\
	u |_{t=0} =u_0,
	\end{array}\right.
	\end{equation}
	where
	\[ L(u) = -\frac{1}{16} \pa_x^3 u + \frac{i}{8} \pa_x^2 u + \frac{1}{2}\pa_x u - \frac{i}{4}\pa_y^2 u + \frac{3}{8}\pa_x\pa_y^2 u,\]
	and the nonlinearity is given by
	\[ N(u)= -\frac{i}{2} |u|^2\, u - \frac{3}{2} |u|^2 \, \pa_x u - \frac{1}{4} u^2 \,\pa_x \overline{u} + \frac{i}{2}\, u\, \pa_x^2 |\nabla|^{-1} (|u|^2).\]

In the formulae above, $u$ is the complex-valued envelope for the modulation wave with complex conjugate $\overline{u} $. This equation was first proposed by Dysthe in \cite{Dysthe}. It can be derived from the incompressible Navier-Stokes equation, after performing an asymptotic expansion of the modulation of a wavetrain. Truncating this approximation at order three would give rise to the cubic NLS equation, which has previously been used for large wave prediction \cite{OnoratoMechanism2}. However, the NLS equation is only valid to model wave spectra with a narrow bandwidth. Dysthe (1979) found that continuing the expansion to fourth-order relaxes the bandwith restriction and improves stability analysis results \cite{roguereview}. In this paper we study the local well-posedness of this equation, with the long-term goal of justifying analytically the numerical results mentioned above.

We immediately note that  a solution to the initial value problem  \eqref{eq:Dysthe} would conserve the mass (i.e. the $L^2$-norm of the initial data), but to the best of our knowledge the equation does not have a conserved energy\footnote{There is a Hamiltonian version of the 1D Dysthe equation \cite{Sulem}, but the existence of a Hamiltonian version in 2D remains an open question.}. Hence, a long time analysis of the solutions of the equation will be difficult to obtain since one would need to conduct the analysis at the $L^2$ level of regularity\footnote{ In this regard it is important to note that in  \cite{FarSap} the authors argue that  their goal is not longtime prediction. In fact their  numerical experimentations  are valid for short to medium time prediction where  one only expects one rogue wave per ocean patch. For a longer period of time rogue waves on different ocean patches may interact and the analysis  becomes certainly more difficult.}. But even for local solutions one expects complications: there is no scaling symmetry available, and there is a strong  anisotropy, which is a byproduct of a preferred direction of propagation assumed during the derivation of the equation. Since a scaling exponent is not available in this case, it is difficult to conjecture for which Sobolev space of data the initial value problem is well posed and for which it is ill posed. In this work, we exhibit results in both directions, although they are not sharp and a gap remains.

 Let us now start by analyzing in more details the challenges one faces in proving well-posedness. Because of the presence of derivatives in the nonlinearity, Strichartz estimates are not enough to close a contraction mapping argument. The fact that we find linear terms such as $\pa_x^3 u$ in the Dysthe equation  is reminiscent of the KdV equation, so one might expect to be able to recover up to two derivatives in $x$ in the nonlinearity using  smoothing effect estimates\footnote{To be precise, we expect to recover one derivative for the linear term $W(t)u_0$. This gain can sometimes be doubled for the Duhamel term using some techniques introduced by Kenig, Ponce and Vega \cite{KPV}.}, which should be enough to close a contraction mapping argument. However, there is an important difference with respect to the KdV equation: the interaction between terms such as $\pa_x^3 u$ and $\pa_x\pa_y^2 u$ which are of opposite sign. This unfavorable interaction translates into an unfavorable cancelation in the dispersive relation, a situation that is also much worse   than  the one given by the  linear operator of the Zakharov-Kuznetsov equation  that involves  $\pa_x\Delta u$ \cite{LinaresPastor}. All this  is better understood by looking at the symbol associated to the linear operator of the Dysthe equation:
\begin{equation} \label{dispersive}\Phi(\xi,\mu)=\frac{1}{16}\xi^3 - \frac{3}{8} \xi \mu^2 - \frac{1}{8}\xi^2 + \frac{1}{4}\mu^2 +\frac{1}{2}\xi.\end{equation}
which will  play an important role  in the rest of the paper.
Based on the work of Kenig, Ponce and Vega on the KdV equation \cite{KPV}, one expects to derive linear smoothing estimates in a space such as $L^{\infty}_x L^2_{t,y}$. A key ingredient in the proof, however, would be  making sure that $\pa_{\xi} \Phi$ does not vanish away from the origin. Unfortunately, this is not the case for us, which is another important difference with respect to the KdV equation. 

A way to overcome this  issue consists of dividing the frequency space into regions where at least one component of $\nabla\Phi$ is nonzero. There will be two such regions in our analysis, one where $\pa_{\xi} \Phi$ does not vanish and another where  $\pa_{\mu} \Phi$ does not. A third, low-frequency region will be necessary as $\Phi$ has two critical points in the coordinate system $(\xi,\mu)$, namely at $(\frac{2}{3},\pm \frac{\sqrt{10}}{3})$. The drawback of this approach is that the space where we recover derivatives when $\pa_{\xi} \Phi$ is nonzero will be $L^{\infty}_x L^2_{t,y}$, while the space corresponding to the region where $\pa_{\mu} \Phi$ is nonzero will be $L^{\infty}_y L^2_{t,x}$. This will in turn give rise to different spaces where matching maximal function estimates are needed, as well as additional complications when estimating the nonlinearity. 

Similar techniques have been used  to study NLS-type equations \cite{KPV4}, and  the Zakharov-Kuznetsov equation \cite{LinaresPastor,LinaresRamos,MolinetPilod,RibaudVento}, with the notable difference that in both these cases the linear operators $\Delta$  and $\partial_x\Delta$ respectively, have better properties, as they allow for a more isotropic division of the frequency space\footnote{Note that the maximal function estimates in the work of Linares and Pastor \cite{LinaresPastor} for the modified Zakharov-Kuznetsov equation admit the same regularity as ours.}. Moreover,  in both cases one can  exploit  scaling symmetries  that  somewhat simplify the analysis. Perhaps a better comparison can be drawn to the work of Kenig and Ziesler on the local well-posedness theory of the Kadomstev-Petviashvili equation \cite{KZ}. This equation is also anisotropic, and its study requires a combination of Strichartz, smoothing and maximal function estimates.  It is also important to recall here that  {\it lateral spaces} such as $L^{\infty}_x L^2_{t,y}$ and $L^{\infty}_y L^2_{t,x}$ introduced above, were a key ingredient in the  study  of Schr\"odinger maps in  the work of Bejeneru-Ionescu-Kenig-Tataru \cite{BIKT}, and also previously introduced  by Bejenaru \cite{Bejenaru}, Bejenaru-Ionescu-Kenig \cite{BIK} and Ionescu-Kenig \cite{IonescuKenig}.

Local well-posedness in $H^{\frac{3}{2}}(\R^2)$ was obtained for the Dysthe equation by Koch and Saut\footnote{After the publication of this paper on arXiv, J.-C. Saut communicated further progress in this problem in an upcoming work \cite{Sautnew}.} in \cite{KochSaut}. Their results are based on a local smoothing estimate and a careful partition of physical space into small cubes, which yield local well-posedness for a large class of equations, including Dysthe's. Our approach is based on exploiting the lateral spaces introduced above, which allows us to obtain a stronger global smoothing effect for the Dysthe equation. This, together with maximal function estimates and Strichartz estimates, yields a local well-posedness theory in $H^{s}(\R^2)$ for $s>1$, as we state below.

%

\subsection{Statement of results}

The main result in this paper is the following:

\begin{thm}\label{thm:maintheorem}
The Dysthe equation \eqref{eq:Dysthe} is locally well-posed for initial data $u_0\in H^s(\R^2)$, $s>1$.
\end{thm}
\begin{rk}
The maximal function estimates that we derive in \Cref{sec:maxfunc} suggest that it might be possible to close the contraction mapping argument, which is used to prove Theorem \ref{thm:maintheorem},  for $s>\frac{3}{4}$. Unfortunately, putting back together the different regions where one can use the smoothing effect creates technical difficulties that we overcome with the Sobolev embedding theorem, hence losing derivatives and forcing us to require   $s>1$.
\end{rk}

\begin{rk}
The proof of Theorem \ref{thm:maintheorem} is conducted using a fixed point theorem on a Banach space built out of norms that come from the proof of the smoothing effect and the maximal function estimates. There is another type of space that one could consider which is a type of $X^{s,b}$ space that first appeared in the context of dispersive equations in the work of Bourgain \cite{Bourgain} and was then further exploited in the work of Kenig-Ponce-Vega \cite{KPV3}. 
Preliminary calculations show that,  although the setup is technically more challenging due to the anisotropic nature of the dispersive relation given by $\Phi$ in \eqref{dispersive}, one could in principle set up  a fixed point approach using the appropriate definition of an $X^{s,b}$  space in this setting. However, this approach would not give  smoothing effect estimates such as the ones in  Section \ref{smoothingeffect}. We believe that this approach may be very fruitful in the periodic case, which is indeed the setup of the numerical study in \cite{FarSap}. In an upcoming work, we are in fact investigating analytically certain questions strictly related to the periodic setting, both in 1D and 2D, that have emerged from the more experimental work in \cite{VandenEijnden} and \cite{FarSap}.
\end{rk}

Despite the fact that there is no scaling symmetry available for this equation, one can gain some intuition about {\it critical regularity} by considering equation \eqref{eq:Dysthe} where we only keep the top order terms, i.e.
\begin{align*}
 \tilde L(u) & = -\frac{1}{16} \pa_x^3 u + \frac{3}{8}\pa_x\pa_y^2 u,\\
\tilde N(u) & = - \frac{3}{2} |u|^2 \, \pa_x u - \frac{1}{4} u^2 \,\pa_x \overline{u} + \frac{i}{2}\, u\, \pa_x^2 |\nabla|^{-1} (|u|^2).
\end{align*}
This new equation enjoys a scaling symmetry, and the homogeneous Sobolev space that remains invariant under it has exponent $s_c=0$. We discuss this connection in more depth in \Cref{sec:illposedness}. In this regard, we present the following result:

\begin{thm}\label{thm:illposedness}
The Dysthe equation \eqref{eq:Dysthe} is ill-posed in $H^s(\R^2)$ whenever $s<0$, in the sense that the initial data-to-solution map, from $H^s (\R^2)$ to $C([0,T], H^s(\R^2))$, is not $C^3$.
\end{thm}

\subsection{Outline}

In Section 2, we develop dispersive and Strichartz estimates. In Section 3, we study the smoothing effect in different frequency regions, and we double the gain for the Duhamel term. In Section 4, we establish maximal function estimates in various spaces. In Section 5, we prove \Cref{thm:maintheorem} using a contraction mapping argument. In Section 6, we prove \Cref{thm:illposedness}. Finally, we have an appendix with some technical results that are useful throughout the paper.

\subsection{Notation}

We will denote by $A\lesssim B$ an estimate of the form $A\leq C B$ for some constant $C$ that might change from line to line. Similarly,  $A\lesssim_d B$ means that the implicit constant $C$ depends on $d$. We will also use the big $\O$ and little $o$ notation, e.g. $A=\O_d(B)$ when $A=\O(B)$ as $d\rightarrow 0$. We write $a-$ to denote the number $a-\varepsilon$ for $0<\varepsilon \ll 1$ small enough. Similarly, we denote by $a+$ the number $a+\varepsilon$ for $0<\varepsilon \ll 1$ small enough.

For $1\leq p,q\leq \infty$ and $u:  [0,T]\times\R^2 \longrightarrow \C$, we define 
\[\norm{u}_{L^p_x L^q_{T,y}}=\left(\int_{\R} \left(\int_{\R} \int_0^T |u(t,x,y)|^q \, dt\,dy\right)^{\frac{p}{q}} \, dx\right)^{\frac{1}{p}} ,\]
with the usual modifications when $p$ or $q=\infty$. For $u: [0,\infty) \times \R^2 \longrightarrow \C$, we will use the notation $L^p_x L^q_{t,y}$ instead, meaning $T=\infty$. We will also write $L^{p}_{T,x,y}$ in the case $p=q$.

We also use the standard notation for the spatial Fourier transform 
\[\widehat{f}(\xi)=\int_{\R} e^{-ix\cdot \xi} \, f(x) \, dx ,\]
as well as $f^{\vee}$ for the inverse Fourier transform. 

We will denote by $\norm{f}_{H^s(\R^2)} =\norm{\langle\nabla\rangle^s f}_{L^2(\R^2)}$ the usual Sobolev norm, where $\langle\nabla\rangle^s$ corresponds to the Fourier multiplier operator with symbol $(1+ |\xi|+|\mu|)^{s}$. We will sometimes use $D_x$ for the Fourier multiplier operator with symbol $\xi$. Finally, we will denote by $C([0,T],H^s (\R^2))$  the space of continuous functions $u$ from a time interval $[0,T]$ to $H^s(\R^2)$ equipped with the norm $\sup_{t\in [0,T]}\norm{u(t,\cdot)}_{H^s(\R^2 )}$.

\subsection{Acknowledgements}

We would like to thank Luis Vega for his useful suggestions and references regarding the smoothing effect. We would also like to thank Mohammad Farazmand and Themistoklis Sapsis for some helpful conversations about their work.

\section{Strichartz estimates}\label{c3:sec:Strichartz}
	
	We will first focus on the linear equation. By taking the Fourier transform with ${(x,y)\mapsto(\xi,\mu)}$, one finds that the solution to the linear equation is:
	\begin{equation}\label{eq:linearflow}
		u(t,x,y)=W(t)u_0(x,y):= \int_{\R^2} e^{ix\xi+iy\mu+i t\,\Phi(\xi,\mu)} \, \widehat{u_0}(\xi,\mu)\, d\xi d\mu,
	\end{equation}
	where
	\begin{equation}\label{eq:phase}
		\Phi(\xi,\mu)=\frac{1}{16}\xi^3 - \frac{3}{8} \xi \mu^2 - \frac{1}{8}\xi^2 + \frac{1}{4}\mu^2 +\frac{1}{2}\xi.
	\end{equation}
 This function has two critical points: $(\frac{2}{3},\pm \frac{\sqrt{10}}{3})$. There is also one zero of the Hessian: $(\frac{2}{3},0)$. As explained in the introduction, this is why we must divide the frequency space into different regions, in which the behavior is quite different.
	
\subsection{Dispersive estimates}
	
	The following four regions in Fourier space will be important to our analysis:
	\begin{align}\label{eq:regions}
	\mathcal{R}_0 & := \left\{ (\xi,\mu)\in\R^2 \mid |\xi|\leq 100 \right\},\\
	\mathcal{R}_1 & := \left\{ (\xi,\mu) \in \R^2 \mid |\xi|>100,\ |\mu|<\frac{|\xi|}{200}\right\},\\
	\mathcal{R}_2 & := \left\{ (\xi,\mu) \in \R^2 \mid |\xi|>100,\ |\mu|>\frac{|\xi|}{200}\right\},\\
	\mathcal{R}_3 & := \mathcal{R}_1  \cup \mathcal{R}_2.
	\end{align}
	
	Fix a function $\psi\in C_c^{\infty}(\R)$ such that $0\leq \psi\leq 1$, $\psi=1$ in the ball $B(0,100)$ and $\mbox{supp}(\psi)\subset B(0,200)$. 
	We define
	\begin{align}\label{eq:cut-offs}
	\chi_0 (\xi) & := \psi(\xi),\\
	\chi_1 (\xi,\mu) & :=  [1-\psi(\xi/\mu)]\cdot [1-\psi (\xi)],\\
	\chi_2 (\xi,\mu) & := \psi(\xi/\mu)\, [1-\psi (\xi)],\\
	\chi_3 (\xi) & := 1-\psi(\xi),.
	\end{align}
	In this way, $\chi_i$ corresponds to the region $\mathcal{R}_i$.
	
	For $k=0,1,2,3$, let us write 
	\[W_k (t)f(x,y):= \int_{\R^2} e^{ix\xi+iy\mu+i t\,\Phi(\xi,\mu)} \, \widehat{f}(\xi,\mu)\,\chi_k (\xi,\mu)\, d\xi d\mu.\]
	
	For the purpose of deriving Strichartz estimates, only two regions matter: $\mathcal{R}_0$ and $\mathcal{R}_3$. However, we will need the subdivision given by $\mathcal{R}_1$ and $\mathcal{R}_2$ in the next section.

We now present dispersive estimates for $W_0(t)$ and $W_3(t)$. The following results were presented in the second author's PhD thesis \cite{Kristin}. Similar results have been obtained by Ben-Artzi-Koch-Saut for a general class of  third-order dispersive equations in 2D \cite{BenArtziKochSaut} (see also \cite{Strich3,Strich2,Strich1,Strich4}). In fact their results yield an overall decay of $|t|^{-2/3}$ in $\R^2$. In this work, we will only exploit the large-frequency estimate \eqref{eq:dispW3} where we have an improved decay of $|t|^{-1}$ (see also Theorem 7.5 in \cite{BenArtziKochSaut}), whose proof we give for completeness.

\begin{prop}
	For $u_0\in L^1_{x,y}$, we have the following dispersive estimates:
		\begin{align}
		\norm{W_0(t) u_0}_{L_{x,y}^\infty} & \lesssim |t|^{-1/2}\norm{u_0}_{L_{x,y}^1},\label{eq:dispW0}\\
		 \norm{W_3(t) u_0}_{L_{x,y}^\infty} & \lesssim |t|^{-1}\norm{u_0}_{L_{x,y}^1}.\label{eq:dispW3}
		\end{align}
\end{prop}
\begin{proof}
	We write $W_k (t)u_0= I_k (t) \ast u_0$ for $k=0,3$, where the convolution is in the variables $x$ and $y$, and
	\[ I_k (t,x,y) = \int e^{ix\,\xi+iy\,\mu-it\,\Phi(\xi,\mu)}\, \chi_k(\xi)\,d\xi\,d\mu.\]

We first prove \eqref{eq:dispW3}, and we will assume that $t>0$ for simplicity. 
	We first perform the $\mu$-integral of $I_3$, i.e.
	\[ f(t,\xi,y):= \int e^{-it\, \left(\frac14-\frac38\xi\right)\, \mu^2+iy\,\mu}\,d\mu\sim \frac{e^{iy^2/[t(3\xi/2-1)]}}{\sqrt{t} \left(\frac32\xi-1\right)^{1/2}},\]
	and rewrite 
	\[
	I_3 (t,x,y)= \int e^{ix\xi-it\left(\frac1{16}\xi^3-\frac18\xi^2+\frac12\xi\right)}\, \chi_3(\xi) \, f(t,\xi,y)\,d\xi=\frac{1}{\sqrt{t}}\int\frac{\chi_3(\xi)}{|1-\frac32\xi|^{1/2}}e^{it\,\psi_{t,x,y}(\xi)}\,d\xi,
	\]
	for the phase
	\begin{equation}\label{psi}
	\psi_{t,x,y}(\xi) = -\frac{1}{16}\, \xi^3+\frac18\, \xi^2-\frac12\, \xi+\frac{x\,\xi}{t}-\left(\frac{|y|}{t}\right)^2\left(\frac32\xi-1\right)^{-1}.
	\end{equation}
	Note that $|1-3\xi/2|\gg 1$ because we are in region $\mathcal{R}_3$, and we are assuming that we are in the case $\xi>100$ as an example.
	
	The critical points of the phase are $\displaystyle{\xi_a=\frac23-Z}$ and $\displaystyle{\xi_b=\frac23+Z}$ where
	\[
	Z = \sqrt2\, \sqrt{-\left(5+\frac{12x}t\right) + \sqrt{\left(5+\frac{12x}t\right)^2+72\left( \frac{|y|}t \right)^2}}.
	\]
We now study whether $\psi_{t,x,y}'(\xi)$ is monotonic in order to apply the Van der Corput lemma. We have
	\begin{equation}\label{psidoubleprime}
	\psi_{t,x,y}''(\xi) = -\frac38\xi+\frac14-\frac92\left(\frac{|y|}{t}\right)^2\left(\frac32\xi-1\right)^{-3}.
	\end{equation}	
	We then have three cases to consider:
	\begin{enumerate}
		\item[(i)] Both $\xi_a$ and $\xi_b$ lie in the region $|\xi|<100$.
		\item[(ii)] Only one of either $\xi_a$ or $\xi_b$ lies in the region $|\xi|>100$ while the other lies in the region $|\xi|<100$.
		\item[(iii)] Both $\xi_a$ and $\xi_b$ lie in the region $|\xi|>100$.
	\end{enumerate}
	
	For Case (i), $\xi_a$ and $\xi_b$ are outside the region in which $\chi_3$ is nonzero and therefore $\psi_{t,x,y}(\xi)$ has no critical points in the region of interest. Since $\xi>100$, one can factor out the term $-\frac38\xi+\frac14$ in \cref{psidoubleprime} to show that
	\[
	|\psi_{t,x,y}''(\xi)|\geq C>0
	\]
	for $C$ independent of $\xi$, and hence $\psi_{t,x,y}'(\xi)$ is monotonic. By \Cref{thm:VanderCorput},
	\[
	\Big | \frac{1}{\sqrt{t}}\int \frac{\chi_3(\xi)}{|1-\frac32\xi|^{1/2}}e^{it\psi_{t,x,y}(\xi)}\,d\xi\Big | \lesssim  \frac{1}{\sqrt{t}} \cdot \frac{1}{t} = t^{-3/2}.
	\]
	for $\xi$ satisfying the scenario in Case (i).
	
	We now turn to Case (ii). Without loss of generality, suppose that $\xi_a$ lies in the region in which $|\xi|>100$ and $\xi_b$ does not. We consider the set
	\[
	\Omega_a=\{\xi : |\xi|>100\text{ and }\xi\notin[\xi_a-\delta,\xi_a+\delta]\}
	\]
	for some $0<\delta\ll1$. For $\xi\in\Omega_a$, the same argument as in Case (i) yields  $|\psi_{t,x,y}''(\xi)|>C$
	for some $C$ independent of $\xi$. Case (iii) is analogous, except we must restrict $\xi\in \Omega_a \cap \Omega_b$. In both cases, $\psi_{t,x,y}'(\xi)$ is monotonic in the corresponding region, and therefore \Cref{thm:VanderCorput} yields:
	\[
	\frac{1}{\sqrt{t}}\int_{\Omega_a \cap \Omega_b} \frac{\chi_3(\xi)}{|1-\frac32\xi|^{1/2}}e^{it\psi_{t,x,y}(\xi)}\,d\xi = \O (t^{-3/2}).
	\]
	Finally, we compute the contributions from the critical points. To do this, we first define the functions $a(\xi)$ and $b(\xi)$ by
\[a(\xi) := \frac{\chi_a(\xi)\chi_3(\xi)}{\sqrt{1-3\xi/2}},\qquad b(\xi) := \frac{\chi_b(\xi)\chi_3(\xi)}{\sqrt{1-3\xi/2}}.\]
	with $\chi_a(\xi)$ being a smooth non-negative function supported at $[\xi_a-\delta,\xi_a+\delta]$ for some $0<\delta\ll1$, and $\chi_b(\xi)$ defined analogously. The  key observation is that $\psi_{t,x,y}''(\xi)$ does not vanish in the region $\mathcal{R}_3$, in fact it admits a uniform lower bound as before. Then we may use \Cref{thm:VanderCorput} (without monotonicity) to conclude that
	\[
	\left|\frac{1}{\sqrt{t}}\int\frac{\chi_a(\xi)\chi_0(\xi)}{|1-\frac32\xi|^{1/2}}e^{it\psi_{t,x,y}(\xi)}\,d\xi + \frac{1}{\sqrt{t}}\int\frac{\chi_b(\xi)\chi_0(\xi)}{|1-\frac32\xi|^{1/2}}e^{it\psi_{t,x,y}(\xi)}\,d\xi\right| \lesssim |t|^{-1}
	\]
	as $t\to\infty$. By combining this with our results over $\Omega_a\cap\Omega_b$, we ultimately obtain $|I_3 (t,x,y)|\lesssim |t|^{-1}$
	as $t\to\infty$. The Young convolution inequality yields \eqref{eq:dispW3}.
	
Finally, we prove  \eqref{eq:dispW0}. Similar to above, we consider
	\[ I_0(t,x,y) = \int e^{-it\Phi(\xi,\mu)+i(x\xi+y\mu)}\chi_0(\xi)\,d\xi\,d\mu= \frac{1}{\sqrt{t}}\int\frac{\chi_0(\xi)}{|1-\frac32\xi|^{1/2}}e^{it\,\psi_{t,x,y}(\xi)}\,d\xi.\]
	This time we take the absolute value inside the integral, and integrate directly to obtain $ |I_0 (t,x,y)|\lesssim  |t|^{-1/2}$. Young's inequality for convolutions then yields \eqref{eq:dispW0}.
\end{proof}

\subsection{Main estimates}

The proofs of the propositions below follow from the well-known results by Keel and Tao \cite{KeelTao}, after interpolating between the conservation of the $L^2$-norm and \eqref{eq:dispW0}-\eqref{eq:dispW3}, respectively.

\begin{prop}[Large frequency Strichartz estimate]\label{2DDystheStrichartzPropLarge} 
	Assume $(q,r)$ and $(\tilde{q},\tilde{r})$ are Strichartz admissible pairs satisfying
	\begin{equation}\label{DystheLargeStrichartz}
	\frac2q = 1-\frac{2}r.
	\end{equation}
	 with $(q,r)\neq (2,\infty)$. Then
		\begin{align}
		\norm{W_3(t)u_0}_{L_t^qL_{x,y}^r}&\lesssim \norm{u_0}_{L_{x,y}^2}\label{Strichlarge1}\\
		\norm{\int_{\R}W_3(t-t')F(t')\,dt'}_{L_t^qL_{x,y}^r}&\lesssim \norm{F}_{L_t^{\tilde{q}'}L_{x,y}^{\tilde{r}'}}.\label{Strichlarge3}
	\end{align}
\end{prop}

\begin{prop}[Small frequency Strichartz estimates]\label{2DDystheStrichartzPropSmall} 
	Assume $(q,r)$ and $(\tilde{q},\tilde{r})$ are Strichartz admissible pairs satisfying
	\begin{equation}\label{DystheSmallStrichartz}
	\frac2q = \frac12-\frac{2}r.
	\end{equation}
	with $(q,r)\neq (2,\infty)$. Then
	\begin{align}
	\norm{W_0(t)u_0}_{L_t^qL_{x,y}^r}&\lesssim \left\|u_0\right\|_{L_{x,y}^2}\label{Strichsmall1}\\
	\norm{\int_{\R}W_0(t-t')F(t')\,dt'}_{L_t^qL_{x,y}^r}&\lesssim \norm{F}_{L_t^{\tilde{q}'}L_{x,y}^{\tilde{r}'}}.\label{Strichsmall3}
	\end{align}
\end{prop}

	\section{Smoothing effect}\label{smoothingeffect}
	\subsection{Introduction}
	In this section ,we divide our frequency space into three regions: $\mathcal{R}_0$, $\mathcal{R}_1$, and $\mathcal{R}_2$, see \eqref{eq:regions}. We start with the high-frequency regions $\mathcal{R}_1$ and $\mathcal{R}_2$, whereas the low-frequency region $\mathcal{R}_0$ will be treated separately at the end of this section.
	
	The reason why we distinguish these regions is because the behavior of $\Phi (\xi,\mu)$ is different in each of them. For instance, the only critical points fall in the region $\mathcal{R}_0$. In $\mathcal{R}_1$ we have that
	\begin{equation}\label{eq:region1}
		|\pa_{\xi} \Phi (\xi,\mu)| = \Big | \frac{3}{16}(\xi^2 - 2 \mu^2) - \frac{1}{4}\xi+\frac{1}{2}\Big | \gtrsim |\xi|^2\quad \mbox{whenever}\quad (\xi,\mu)\in \mathcal{R}_1.
	\end{equation}
	In the other region, $\mathcal{R}_2$, this lower bound is false and in fact this partial derivative can vanish. However, we have that 
	\begin{equation}\label{eq:region2}
		|\pa_{\mu} \Phi (\xi,\mu)| = \Big | -\frac{3}{4}\,\xi\,\mu - \frac{1}{2}\mu\Big | \gtrsim |\xi| \, |\mu|\quad \mbox{whenever}\quad (\xi,\mu)\in \mathcal{R}_2.
	\end{equation}
	These features will give rise to different smoothing effects in different frequency regions, which is another example of the anisotropic nature of the Dysthe equation.
	
	\subsection{Large frequency smoothing effect}
	
	For $k=1,2$, let us write 
	\[W_k (t)f(x,y):= \int_{\R^2} e^{ix\xi+iy\mu+i t\,\Phi(\xi,\mu)} \, \widehat{f}(\xi,\mu)\,\chi_k (\xi,\mu)\, d\xi d\mu.\]
	
	Then we have
	\begin{prop}\label{thm:linearsmoothing}
		\begin{align}
			\norm{D_x\, W_1 (t) f}_{L^{\infty}_x L^2_{t,y}} & \lesssim \norm{f}_{L^2_{x,y}}, \label{eq:linearsmoothing1}\\
			\norm{D_x\, W_2 (t) f}_{L^{\infty}_y L^2_{t,x}} & \lesssim \norm{f}_{L^2_{x,y}}.\label{eq:linearsmoothing2}
		\end{align}
	\end{prop}
	\begin{rk}
	Note also that the same result and proof are true for $P\, D_x$ instead of $D_x$, where $P$ is any Fourier multiplier operator whose symbol is bounded. In this regard, see Theorem 4.1  in \cite{KPV2} for more general results. We will use this important remark for the operator $\pa_x |\nabla|^{-1}$ in \Cref{sec:contraction}.
	\end{rk}
	\begin{proof}
		 We only prove this result for $W_1$ since the argument is analogous. By the Plancherel theorem,
		\[ \norm{D_x\, W_1 (t) f}_{L^2_{t,y}}= \norm{\int_{\R}  e^{ix\xi+i t\,\Phi(\xi,\mu)} \, |\xi|\,\widehat{ f}(\xi,\mu)\,\chi_1(\xi,\mu)\, d\xi}_{L^2_{t,\mu}}\]
		Now we do the change of variables $\tilde{\xi}=\Phi(\xi,\mu)$ whose Jacobian is $J(\xi,\mu)\sim |\xi|^{-2}$.
		\begin{multline*}
			\int_{\R}  e^{ix\xi+i t\,\Phi(\xi,\mu)} \, |\xi|\,\widehat{f}(\xi,\mu)\, d\xi= \\
			\int_{\R}  e^{ix\Phi^{-1}_{\mu}(\tilde{\xi})+i t\,\tilde{\xi}} \, |\Phi^{-1}_{\mu}(\tilde{\xi})|\,\widehat{f}(\Phi^{-1}_{\mu}(\tilde{\xi}),\mu)\, \chi_1(\Phi^{-1}_{\mu}(\tilde{\xi}),\mu)\, J(\Phi^{-1}_{\mu}(\tilde{\xi}),\mu) \, d\tilde{\xi}.
		\end{multline*}
		We take the $L^2_t$-norm of the above. The Plancherel theorem yields
		\begin{align*} \norm{D_x \, W_1 (t) f}_{L^2_{t,y}}^2 & \lesssim \int_{\R^2} |\Phi^{-1}_{\mu}(\tilde{\xi})|^2\,|\widehat{f}(\Phi^{-1}_{\mu}(\tilde{\xi}),\mu)|^2\, |\chi_1(\Phi^{-1}_{\mu}(\tilde{\xi}),\mu)|^2\, J(\Phi^{-1}_{\mu}(\tilde{\xi}),\mu)^2 \, d\tilde{\xi} \, d\mu\\
			& = \int_{\R^2} |\xi|^2\,|\widehat{f}(\xi,\mu)|^2\, |\chi_1(\xi,\mu)|^2\, J(\xi,\mu) \, d\xi \, d\mu \lesssim \int_{\R^2}|\widehat{f}(\xi,\mu)|^2\, d\xi \, d\mu=\norm{f}_{L^2_{x,y}}^2.
		\end{align*}
		We finish by taking the supremum in $x$.
	\end{proof}
	
The following corollary is obtained by writing the dual estimate to \eqref{eq:linearsmoothing1} and using the fact that $W(t)$ is unitary in $L^2$.
	
	\begin{cor} With $L^p_T= L^p_t ([0,T])$ for $1\leq p\leq \infty$, we have the following estimates:
		\begin{align*}
			\norm{\pa_x\, \int_0^t  W_1 (t-t') F(t')\, dt'}_{L^{\infty}_x L^2_{T,y}} & \lesssim \norm{F}_{L^1_T L^2_{x,y}},\\
			\norm{\pa_x\, \int_0^t  W_2 (t-t') F(t')\, dt'}_{L^{\infty}_y L^2_{T,x}} & \lesssim \norm{F}_{L^1_T L^2_{x,y}}.
		\end{align*}
	\end{cor}
	
\subsection{Additional linear estimates}
	
	To control the evolution at low-frequencies, we will need a combination of the smoothing effect and Strichartz estimates. We define $P_0$ to be the Fourier multiplier operator corresponding to the symbol $\chi_0(\xi)$, defined in \eqref{eq:cut-offs}.
	
		\begin{prop}\label{thm:lowfreq} 
		For any $s\in\R$, we have the following estimates:
		\begin{align*}
			\norm{\pa_x\, \langle\nabla \rangle^s P_0 u}_{L^{\infty}_x L^2_{t,y}} & \lesssim \norm{\langle\nabla \rangle^s u}_{L^2_{t,x,y}},\\
			\norm{\pa_x P_0 u}_{L^2_t L^{\infty}_{x,y}} & \lesssim \norm{\langle\nabla\rangle^{\frac{1}{2}+} u}_{L^2_{t,x,y}}.
		\end{align*}
	\end{prop}
	\begin{proof}
		\begin{enumerate}
			\item For the first estimate we write 
			\begin{equation}\label{eq:P0}
				\pa_x\, \langle\nabla\rangle^s P_0 u = \int_{|\xi|\leq 100} e^{ix\xi}\, \xi  \, \widehat{\langle\nabla\rangle^s\, u}(\xi,t,y)\, \chi_0(\xi)\, d\xi.
			\end{equation}
			We take the $L^2_{t,y}$ norm using the Minkowski inequality:
			\[ \norm{\pa_x \langle\nabla\rangle^s P_0 u }_{L^2_{t,y}} \lesssim  \int_{|\xi|\leq 100} \norm{\widehat{\langle\nabla\rangle^s\, u}(\xi)}_{L^2_{t,y}}\, d\xi.\]
			Finally, we use the Cauchy-Schwartz inequality and the Plancherel theorem:
			\[ \norm{\pa_x \langle\nabla\rangle^s P_0 u}_{L^{\infty}_x L^2_{t,y}} \lesssim \norm{\langle\nabla\rangle^s\, u}_{L^2_{t,x,y}}.\]
			\item By the Minkowski inequality, we take the $L^2_t L^{\infty}_{x,y}$-norm of \eqref{eq:P0} (with $s=0$) and use the Holder inequality to obtain
			\[ \norm{\pa_x \, P_0 u}_{L^2_t L^{\infty}_{x,y}} \lesssim \norm{\widehat{u}}_{L^2_t L^2_{\xi} L^{\infty}_y}\lesssim \norm{\langle \pa_y\rangle^{\frac{1}{2}+} \widehat{u}}_{L^2_t L^2_{\xi} L^2_y}\lesssim \norm{\langle\nabla\rangle^{\frac{1}{2}+} u}_{L^2_{t,x,y}}.\]
			The Sobolev inequality and the Plancherel inequality give the last steps.
		\end{enumerate}
	\end{proof}
	
 Because Strichartz estimates are not available for the pair $(q,r)=(2,\infty)$, we use the Sobolev embedding theorem to get as close as necessary to this space.

	\begin{lem}\label{thm:Strichartz2}
		For $\varepsilon>0$ small enough, consider the admissible pair $r=\frac{2}{\varepsilon}$ and $q=\frac{2}{1-\varepsilon}$. Then
		\[ \norm{u}_{L^2_T L^{\infty}_{x,y}} \lesssim T^{\varepsilon/2}\, \norm{\langle\nabla\rangle^{\varepsilon} u}_{L^q_T L^r_{x,y}}.\]
	\end{lem}
	\begin{proof}
		The Sobolev embedding theorem in $\R^2$ allows us to go from $L^{\infty}_{x,y}$ to $W^{\varepsilon,r}_{x,y}$. Then the H\"older inequality in time allows us to go from $L^2_T$ to $L^q_T$ and pick up a factor of $T^{\varepsilon/2}$.
	\end{proof}
	
	\subsection{The Hilbert transform method}\label{sec:Hilbert}
	
	In the remainder of this section, we show how to double the smoothing effect for the Duhamel term. 
In order to do that, we present an important quantity, analogous to the one defined in \cite{KPV}.
	\begin{equation}\label{eq:fullFT}
		\widetilde{W}(t)f(x,y)=\int_{\R^3} e^{ix\xi+iy\mu+it\tau}\, \frac{\widehat{f}(\tau,\xi,\mu)}{\tau - \Phi (\xi,\mu)}\, d\xi d\mu d\tau.
	\end{equation}
	Note that here $\widehat{f}$ denotes the Fourier transform in all three variables $t,x,y$. We will also denote by $\widehat{f}^{(x)}$ the Fourier transform in the variable $x$.
	
	This quantity $\widetilde{W}$ appears as an ansatz of the linear Dysthe equation with an inhomogeneous term $f$, after taking the Fourier transform both in time and space. 
	
	Recall the definition of Hilbert transform as a Fourier multiplier $\widehat{Hg}(\xi)=\mbox{sign}(\xi)\, \widehat{g}(\xi)$.
Our operator $\widetilde{W}$ can be interpreted precisely as a Hilbert transform. Formally:
	\[ \widetilde{W}(t)f(x,y)= \int_{\R^2}  e^{ix\xi+iy\mu}\, (Hg)_{t,\xi,\mu}\left(\Phi(\xi,\mu)\right)\, d\xi d\mu,\]
	where
	\[ g(\tau; t,\xi,\mu)=e^{it\tau}\,\widehat{f}(\tau,\xi,\mu).\]
	An equivalent expression would be:
	\[ \widetilde{W}(t)f(x,y)=\int_{\R^3} e^{ix\xi+iy\mu+i\tilde{\tau} \Phi(\xi,\mu)}\, \mbox{sign}(\tilde{\tau})\, (\mathcal{F}_{\tau}g)(\tilde{\tau};t,\xi,\mu)\, d\tilde{\tau} d\xi d\mu.\]
	Note that $(\mathcal{F}_{\tau}g)(\tilde{\tau};t,\xi,\mu)= \widehat{f}^{(x,y)}(-\tilde{\tau}+t,\xi,\mu)$ where $\widehat{f}^{(x,y)}$ is only the Fourier transform of $f$ with respect to $x$ and $y$. Using this together with the change of variables $t'=t-\tilde{\tau}$, one may show that 
	\begin{equation}
	\label{eq:relationshipW}
		\widetilde{W}(t)f(x,y) = \, 2\int_{0}^t W(t-t')f(t',\cdot)\, dt' -\int_{\R} W(t-t')f(t',\cdot)\, dt' +2\,\int_{-\infty}^0 W(t-t')f(t',\cdot)\, dt'.
	\end{equation}
Thus in order to understand $\int_{0}^t W(t-t')f(t',\cdot)\, dt'$, it is enough to study $\widetilde{W}(t)f$.

In order to justify the formal computations above, we introduce for $k=1,2$:
	\[\widetilde{W}_{k,\varepsilon}(t)f(x,y)=\int_{\varepsilon<|\tau - \Phi(\xi,\mu)|<1/\varepsilon} e^{ix\xi+iy\mu+it\tau}\, \frac{\widehat{f}(\tau,\xi,\mu)}{\tau - \Phi (\xi,\mu)}\,\chi_k(\xi,\mu)\, d\xi\, d\mu\,  d\tau.\]
We  define the Fourier multiplier operator $Q_j$ for $j\in \N$, which corresponds to 
	\[ \widehat{Q_j f}(\xi,\mu)= m_j (\xi,\mu)\, \widehat{f}(\xi,\mu),\]
where $m_j \in C_c^{\infty}(\R^2)$, and 
	\begin{equation}\label{eq:defQj}
		m_j (\xi,\mu)=\left\lbrace\begin{array}{ll}
			1 & \mbox{if}\ 2^{j}\leq \sqrt{\xi^2+\mu^2}\leq 2^{j+1},\\
			0 & \mbox{unless}\ 2^{j-1}\leq \sqrt{\xi^2+\mu^2}\leq 2^{j+2}.
		\end{array}\right.
	\end{equation}
	
The following results, as well as the proof, are based on the work of Kenig, Ponce and Vega for the KdV equation, see \cite{KPV}.  But because the linear operator is two dimensional and anisotropic, this proof is much more technical.

	\begin{prop}\label{thm:1} For any $f\in C_c^{\infty}(\R)_t \otimes C_c^{\infty}(\R)_x \otimes C_c^{\infty}(\R)_y$ and all $(t,x,y)\in\R^3$, we have that 
		\[ \lim_{\varepsilon\rightarrow 0} \widetilde{W}_{k,\varepsilon}(t)Q_jf(x,y) = \int_{\R^2} \left(\lim_{\varepsilon\rightarrow 0}
		\int_{\varepsilon<|\tau - \Phi(\xi,\mu)|<1/\varepsilon} e^{ix\xi+iy\mu+it\tau}\, \frac{\widehat{Q_jf}(\tau,\xi,\mu)}{\tau - \Phi (\xi,\mu)}\,d\tau\right)\, \chi_k (\xi,\mu)\, d\xi d\mu\]
		for any $j\in \N$ and any $k=1,2$.
	\end{prop}
	\begin{proof}
		We write $f(t,x,y)=w(t)\, v(x,y)$ with $w\in C_c^{\infty}(\R)$ and $v\in C_c^{\infty}(\R)\otimes C_c^{\infty}(\R)$. We set $h_t(\tau)= e^{it\tau} \widehat{w}(\tau)$ and write
		\[ \lim_{\varepsilon\rightarrow 0}
		\int_{\varepsilon<|\tau - \Phi(\xi,\mu)|<1/\varepsilon} e^{ix\xi+iy\mu+it\tau}\, \frac{\widehat{Q_jf}(\tau,\xi,\mu)}{\tau - \Phi (\xi,\mu)}\,d\tau=\widehat{Q_jv}(\xi,\mu) \, H(h_t)(\Phi(\xi,\mu)).\]
		Since $h_t\in C_c^{\infty}(\R)$, $H(h_t)\in H^s(\R)$ for all $s>0$ and thus $H(h_t)\in L^{\infty}(\R)$. Therefore,
		\[\int_{\R^2} e^{ix\xi+iy\mu} \, \widehat{Q_jv}(\xi,\mu)\, H(h_t)(\Phi(\xi,\mu))\, \chi_k (\xi,\mu)\, d\xi d\mu\]
		is absolutely convergent.
		
		Let 
		\begin{equation}\label{eq:maximalH}
			H^{\ast}(g)(x):=\sup_{\varepsilon>0} | H_{\varepsilon}g(x)|=\sup_{\varepsilon>0} \Big | \int_{\varepsilon < |x-y|<\varepsilon^{-1}} \frac{g(y)}{x-y}\, dy\Big|.
		\end{equation}
		
		In order to prove the proposition we need to show that 
		\[ \lim_{\varepsilon\rightarrow 0} \Big | \int_{\R^2} e^{ix\xi+iy\mu} \, \widehat{Q_jv}(\xi,\mu)\, \left[ H(h_t)(\Phi(\xi,\mu))- H_{\varepsilon}(h_t)(\Phi(\xi,\mu))\right]\, \chi_k (\xi,\mu)\,d\xi d\mu\Big | = 0.\]
		
		But recall that $H_{\varepsilon}g\rightarrow Hg$ pointwise, $|H_{\varepsilon}g(x)-Hg(x)|\leq 2 H^{\ast}g(x)$ and that $H^{\ast}:L^p \rightarrow L^p$ continuously for $1<p<\infty$.
		
		We finish by using the dominated convergence theorem thanks to the fact that 
		\[ \widehat{Q_jv}\, H^{\ast}(h_t)(\Phi)\, \chi_k\in L^1(d\xi d\mu).\]
		
		In order to justify this last fact we do a change of variables. Assume $k=1$ for simplicity, since the proof is analogous in the other case. We set $\tilde{\xi}=\Phi(\xi,\mu)$, $\tilde{\mu}=\mu$, whose Jacobian is $J(\xi,\mu)\sim |\xi|^{-2}$. Then we use the Cauchy-Schwartz inequality twice:
		\begin{align*}
			\left(\int |\widehat{Q_jv}(\xi,\mu)| |H^{\ast}(h_t)(\Phi(\xi,\mu))| \, \chi_1 (\xi,\mu)\,d\xi d\mu\right)^2 & = \left(\int |\widehat{Q_jv}(\tilde{\xi},\tilde{\mu})| |H^{\ast}(h_t)(\tilde{\xi})| \, J(\tilde{\xi},\tilde{\mu})\,\chi_1 \, d\tilde{\xi} d\tilde{\mu}\right)^2\\
			& \hspace{-4cm}\lesssim \norm{H^{\ast}(h_t)}_{L^2_{\tilde{\xi}}}^2 \,  \int_{\tilde{\xi}} \left( \int_{\tilde{\mu}} |\widehat{Q_jv} \, J\,\chi_1 | \right)^2  \lesssim \norm{h_t}_{L^2_{\tau}}^2 \, \int_{\tilde{\xi}} 2^j\,|\widehat{Q_jv}|^2\, |J|^2\, \chi_1^2 \\
			& \hspace{-4cm} \lesssim \norm{w}_{L^2_{t}}^2 \, \norm{Q_j v}_{L^2_x L^2_y}^2 = \norm{Q_j f}_{L^2_{x,y,t}}^2.
		\end{align*} 
	\end{proof}
	
	The following result guarantees that, after taking two derivatives, we can integrate in one single variable.
	
	\begin{prop}\label{thm:2} For $f\in C_c^{\infty}(\R)_{t}\otimes C_c^{\infty}(\R)_x \otimes C_c^{\infty}(\R)_{y}$ and all $(t,x,y)\in\R^3$, we have that 
		\begin{multline}\label{eq:kernel1}
			\lim_{\varepsilon\rightarrow 0} 
			\int_{\varepsilon<|\tau - \Phi(\xi,\mu)|<1/\varepsilon} e^{ix\xi+iy\mu+it\tau}\, \frac{\xi^2}{\tau - \Phi (\xi,\mu)}\,\widehat{Q_j f}(\tau,\xi,\mu)\, \chi_1 (\xi,\mu)\,d\xi d\mu d\tau=\\
			\int_{\R^2} e^{it\tau+i y\mu}\left(\lim_{\varepsilon\rightarrow 0}
			\int_{\varepsilon<|\tau - \Phi(\xi,\mu)|<1/\varepsilon} e^{ix\xi}\, \frac{\xi^2}{\tau - \Phi (\xi,\mu)}\,\widehat{Q_j f}(\tau,\xi,\mu)\,\chi_1 (\xi,\mu)\,d\xi\right) \, d\mu d\tau.
		\end{multline}
		and 
		\begin{multline}\label{eq:kernel2}
			\lim_{\varepsilon\rightarrow 0} 
			\int_{\varepsilon<|\tau - \Phi(\xi,\mu)|<1/\varepsilon} e^{ix\xi+iy\mu+it\tau}\, \frac{\xi\,\mu}{\tau - \Phi (\xi,\mu)}\,\widehat{Q_j f}(\tau,\xi,\mu)\, \chi_2 (\xi,\mu)\, d\xi\, d\mu\, d\tau=\\
			\int_{\R^2} e^{it\tau+i x\xi}\left(\lim_{\varepsilon\rightarrow 0}
			\int_{\varepsilon<|\tau - \Phi(\xi,\mu)|<1/\varepsilon} e^{iy\mu}\, \frac{\xi\, \mu}{\tau - \Phi (\xi,\mu)}\,\widehat{Q_j f}(\tau,\xi,\mu)\,\chi_2 (\xi,\mu)\,d\mu\right) \, d\xi\,d\tau.
		\end{multline}
	\end{prop}
	\begin{proof}
		Let us do the case of $\chi_1$ only, the other one being analogous. The existence of the first limit follows from \Cref{thm:1}. We set $f(t,x,y)=v(x)\, w(t,y)$ for $v\in C_c^{\infty}(\R)$ and $w\in C_c^{\infty}(\R)\otimes C_c^{\infty}(\R)$ and let 
		\begin{align}
			\mathcal{K}_{\varepsilon} \widehat{v}(\mu,\tau) & :=\int_{\varepsilon<|\tau - \Phi(\xi,\mu)|<1/\varepsilon} \frac{\xi^2}{\tau - \Phi (\xi,\mu)}\,\widehat{v}(\xi)\,\chi_1(\xi,\mu)\, m_j (\xi,\mu)\, d\xi,\nonumber\\
			\mathcal{K}\widehat{v}(\mu,\tau) & := \lim_{\varepsilon\rightarrow 0} \mathcal{K}_{\varepsilon}\widehat{v}(\mu,\tau),\qquad \mathcal{K}^{\ast}\widehat{v}(\mu,\tau) := \sup_{\varepsilon>0} |\mathcal{K}_{\varepsilon}\widehat{v}(\mu,\tau)|.\label{eq:K1}
		\end{align}
		If we show that for $\widehat{v}\in\mathcal{S}(\R)$ we have the pointwise convergence given in \eqref{eq:K1}, for almost every $\mu$ and $\tau$, and that $\mathcal{K}^{\ast}:L^{p}_{\xi}\rightarrow L^{q}_{\mu}L^{p}_{\tau} $ continuously for some $p,q\in (1,\infty)$, then we may finish the proof with the dominated convergence theorem, as in \Cref{thm:1}.
		
		We do the change of variables $\tilde{\xi}=\Phi_{\mu} (\xi)=\Phi (\xi,\mu)$, whose Jacobian is $|\det(\nabla\Phi_{\mu})(\xi)|\sim |\xi|^2$.
		Since the Jacobian is nonzero in $\mathcal{R}_1$, we have a differentiable bijection between $\Omega_{\mu,\tau}^{\varepsilon}\cap \pi(R_1)$ and its image, where $\pi$ is the projection $\pi (\xi,\mu)=\xi$. We may even check that:
		\[\Phi_{\mu} (\Omega_{\mu,\tau}^{\varepsilon}\cap \pi(\mathcal{R}_1))=\{\tilde{\xi}\in \R \mid \varepsilon<|\tau - \tilde{\xi}|<1/\varepsilon\}\cap \Phi_{\mu}(\pi(\mathcal{R}_1)).\]
		Then we have
		\begin{multline*}
		\mathcal{K}_{\varepsilon} \widehat{v}(\mu,\tau) = \int_{\Phi_{\mu} (\Omega_{\mu,\tau}^{\varepsilon}\cap\pi(\mathcal{R}_1))} \frac{\Phi_{\mu}^{-1}(\tilde{\xi})^2}{\tau - \tilde{\xi}}\,\widehat{v}\left(\Phi_{\mu}^{-1}(\tilde{\xi})\right)\,\Big |\det(\nabla \Phi_{\mu}^{-1})(\tilde{\xi})\Big | \, \chi_1 \, m_j\, d\tilde{\xi}\\
		= H_{\varepsilon}(\chi_{\Phi_{\mu}(\pi(\mathcal{R}_1))}\,G_{\mu}),
		\end{multline*}
		where we recall that 
		\begin{align*}
			H_{\varepsilon}(G_{\mu})(\mu,\tau) & := \int_{\varepsilon<|\tau - \tilde{\xi}|<1/\varepsilon} \ \frac{1}{\tau-\tilde{\xi}} \, G_{\mu}(\tilde{\xi})\, d\tilde{\xi},\\
			G_{\mu} (\tilde{\xi}) & := \Phi_{\mu}^{-1}(\tilde{\xi})^2\, \,\widehat{v}\left(\Phi_{\mu}^{-1}(\tilde{\xi})\right)\,\Big |\det(\nabla \Phi_{\mu}^{-1})(\tilde{\xi})\Big |\,\chi_1 (\Phi_{\mu}^{-1}(\tilde{\xi}),\mu) \, m_j (\Phi_{\mu}^{-1}(\tilde{\xi}),\mu),
		\end{align*}
		and $\chi_{\Phi_{\mu}(\pi(\mathcal{R}_1))}$ is the characteristic function of the set $\Phi_{\mu}(\pi(\mathcal{R}_1))$.
		
		Remember that $H_{\varepsilon}$ and $H^{\ast}$ were defined in \eqref{eq:maximalH}. Note that as $\varepsilon\rightarrow 0$, $\Omega_{\mu,\tau}^{\varepsilon}\cap \pi(\mathcal{R}_1)\rightarrow \pi(\mathcal{R}_1)$. In order to finish, it would be enough to show that 
		\[ g \mapsto H^{\ast}\left( \chi_{\Phi_{\mu}(\pi(\mathcal{R}_1))} G_{\mu} \right)\]
		is a continuous map. This follows from the properties of the Hilbert transform, since $H^{\ast}$ maps $L^2_{\tilde{\xi}}$ to $L^2_{\tau}$ uniformly in $\mu$.
	\end{proof}
	
		\begin{rk}
	Since $C_c^{\infty}(\R)_t \otimes C_c^{\infty}(\R)_x \otimes C_c^{\infty}(\R)_y$ is dense in $L^p_{x,y} L^q_T$ and $L^q_T L^p_{x,y}$ for any $p,q\in [1,\infty)$, \Cref{thm:1} and \Cref{thm:2} extend to those spaces.
	\end{rk}
	
	
	\subsection{Doubling the smoothing effect}
	
	In the remainder of this section we will show how to double the smoothing effect in the region $\mathcal{R}_1$. The case of the region $\mathcal{R}_2$ is simpler, since the function $\Phi(\xi,\mu)$ has degree 2 when regarded as a function of the second variable $\mu$, for fixed $\xi$.  Analogous results were obtained in \cite{KPV,KochSaut}, among others. In particular, Lemma 2.1 and Lemma 2.2 in \cite{KochSaut} capture the fundamental ideas behind such smoothing effect, which were already present in the work of H\"{o}rmander \cite{Hormander}. Their estimates are for kernels defined globally in $\R^2$, whereas we will develop specific estimates tailored to our frequency-dependent cutoffs. It is therefore important and nontrivial to make sure that the $L^{\infty}$ bounds present in our estimates do not depend on any of the variables involved  in the cutoffs.
	
	We start with some preliminary lemmata:
	
	\begin{lem}\label{thm:zeroes}
		For fixed $\mu,\tau$, consider the solutions to the equation
		\[ \tau - \Phi (\xi,\mu)= 0\quad \mbox{for}\ (\xi,\mu)\in\mathcal{R}_1.\]
		Then there are at most two solutions $\xi_j (\tau,\mu) \in \mathcal{R}_1$ for $j=0,1$. Moreover, if one such solution exists, say $\xi_0$, then there exists some $\varepsilon>0$ independent of $\mu$ and $\tau$ such that 
		\[ \Phi : B(\xi_0, \varepsilon )\rightarrow B\left(\Phi (\xi_0,\mu),\frac{\pa_{\xi}\Phi (\xi_0,\mu)}{2}\, \varepsilon \right) \]
		is a $C^1$ diffeomorphism. In fact, one can take $\varepsilon=|\xi_0|\, 10^{-6}$.
	\end{lem}
	\begin{proof}
		The fact that there are at most two solutions in $\mathcal{R}_1$ follows from the fact that $\pa_{\xi} \Phi (\cdot,\mu)$ is positive and $\Phi(\cdot, \mu)$ is a third order polynomial.
		
		The existence of such an $\varepsilon>0$ (but perhaps dependent on $\mu,\tau$) follows from \Cref{thm:inversefunction}. According to this result, $\varepsilon$ can be chosen so that 
		\begin{equation}\label{eq:chooseepsilon} \Big |\frac{\pa_{\xi}\Phi (\xi,\mu)}{\pa_{\xi}\Phi (\xi_0,\mu)} \Big | \geq \frac{1}{2}\quad  \mbox{for all}\ \xi\in B(\xi_0,\varepsilon).
		\end{equation}
		 By taking $\varepsilon=|\xi_0|\, 10^{-6}$ one can guarantee that condition \eqref{eq:chooseepsilon} is satisfied.
	\end{proof}
	
	For each solution $\xi_0(\tau,\mu)$ we consider a function $\varphi\in C_c^{\infty}(\R)$ supported in $B(\xi_0,\varepsilon)$, with $0\leq \varphi\leq 1$ and such that $\varphi=1$ in $B(\xi_0,\varepsilon/2)$. If $(\xi,\mu)\in \mathcal{R}_1$, then $\xi\in (-\infty,-a_{\mu}) \cup (a_{\mu},+\infty)$ where 
	\[ a_{\mu}=\max \{ 100,\, 200\,|\mu|\},\] 
	which gives rise to two connected components. Furthermore, there can be a maximum of one solution to \eqref{eq:levelcurves} in $(\frac{a_{\mu}}{2},+\infty)$ and another in $(-\infty, -\frac{a_{\mu}}{2})$. This is due to the fact that $1\lesssim \pa_{\xi} \Phi (\xi,\mu)$ there, see \eqref{eq:region1}.
	
	\begin{lem}\label{thm:zeroes2}
		For fixed $\mu,\tau$, consider the equation
		\begin{equation}\label{eq:levelcurves}
			\tau - \Phi (\xi,\mu)= 0\quad \mbox{for}\ \xi\in\R.
		\end{equation}
		Suppose that $\xi_0$ is a solution to \eqref{eq:levelcurves} in $(\frac{a_{\mu}}{2},+\infty)$, then for all $\xi\in(a_{\mu},+\infty)$ we have that
		\[ |\tau - \Phi (\xi,\mu)|\geq c \, |\xi|^2 \, |\xi-\xi_0|,\] 
		where $c$ is independent of $\xi,\tau,$ and $\mu$. If there is no solution to \eqref{eq:levelcurves} in $(\frac{a_{\mu}}{2},+\infty)$, then for all $\xi\in(a_{\mu},+\infty)$ we have
		\[ |\tau - \Phi (\xi,\mu)|\geq c \, |\xi|^3\]
		for some independent constant $c$. Analogous statements hold regarding $(-\infty,-a_{\mu})$.
	\end{lem}
	\begin{proof}
		Suppose for example that $\xi\in (a_{\mu},+\infty)$. If there are no solutions to \eqref{eq:levelcurves} in $(\frac{a_{\mu}}{2},+\infty)$, then the minimum of $\Phi$ in $(\frac{a_{\mu}}{2},+\infty)$ is achieved precisely at $\frac{a_{\mu}}{2}$. Note also that $\Phi(\xi,\mu)-\tau$ must then be positive in $(\frac{a_{\mu}}{2},+\infty)$. Then by \eqref{eq:region1},
		\begin{align*}
			\Phi(\xi,\mu)-\tau &\geq \Phi(\xi,\mu)- \Phi(\frac{a_{\mu}}{2},\mu)=\int_{a_{\mu}/2}^{\xi} \pa_{\xi} \Phi(z, \mu)\, dz\\
			& \gtrsim\int_{a_{\mu}/2}^{\xi} z^2\, dz= \xi^3- \left(\frac{a_{\mu}}{2}\right)^3 \gtrsim \xi^3.
		\end{align*}
		The last inequality follows from the fact that $\xi>a_{\mu}$. A similar argument proves the corresponding lower bound if there is a solution  to \eqref{eq:levelcurves} in $(\frac{a_{\mu}}{2},+\infty)$.
	\end{proof}
	
	We are finally in a position to prove our main result.
	
	\begin{thm}\label{thm:3}
		For $f\in C_c^{\infty}(\R^3)$ and all $(t,x,y)\in\R^3$, we have that 
		\begin{multline*}
			\lim_{\varepsilon\rightarrow 0}
			\int_{\varepsilon<|\tau - \Phi(\xi,\mu)|<1/\varepsilon} e^{ix\xi}\, \frac{\xi^2}{\tau - \Phi (\xi,\mu)}\,\widehat{Q_jf}(\tau,\xi,\mu)\,\chi_1 (\xi,\mu)\, d\xi\\
			=\int_{\R} K_1 (\tau, x-\tilde{x},\mu)\, \widehat{Q_jf}^{(t,y)}(\tau,\tilde{x},\mu) \, d\tilde{x},
		\end{multline*}
		where $\widehat{Q_jf}^{(t,y)}$ denotes the Fourier transform in the variables $t,y$ only, and 
		\[ K_1 (\tau,x,\mu)=\lim_{\varepsilon\rightarrow 0}\int_{\varepsilon<|\tau - \Phi(\xi,\mu)|<1/\varepsilon} e^{ix\xi}\, \frac{\xi^2}{\tau - \Phi (\xi,\mu)}\,\chi_1 (\xi,\mu)\, d\xi\]
		where the limit exists for every $(\tau,x,\mu)$ and $K_1\in L^{\infty}_{\tau, x,\mu}(\R^3)$.
		
		Similarly, we have
		\begin{multline*}
			\lim_{\varepsilon\rightarrow 0}
			\int_{\varepsilon<|\tau - \Phi(\xi,\mu)|<1/\varepsilon} e^{iy\mu}\, \frac{\xi\,\mu}{\tau - \Phi (\xi,\mu)}\,\widehat{Q_jf}(\tau,\xi,\mu)\,\chi_2 (\xi,\mu)\, d\mu\\
			=\int_{\R} K_2 (\tau, \xi,y-\tilde{y})\, \widehat{Q_jf}^{(t,x)}(\tau,\xi,\tilde{y}) \, d\tilde{y},
		\end{multline*}
		where
		\[ K_2 (\tau,\xi,y)=\lim_{\varepsilon\rightarrow 0}\int_{\varepsilon<|\tau - \Phi(\xi,\mu)|<1/\varepsilon} e^{iy\mu}\, \frac{\xi\,\mu}{\tau - \Phi (\xi,\mu)}\,\chi_2 (\xi,\mu)\, d\mu\]
		and the limit exists for every $(\tau,\xi,y)$ and $K_2 \in L^{\infty}_{\tau, \xi,y}(\R^3)$.
	\end{thm}
	\begin{proof}[Proof for the first kernel]
		Fix $\tau$ and $\mu$, and let us write 
		\[ K_1(\tau, x, \mu)=\lim_{\delta\rightarrow 0}\, \int_{\delta<|\tau- \Phi(\xi,\mu)|<\delta^{-1}} e^{ix\xi}\, \frac{\xi^2}{\tau- \Phi(\xi,\mu)} \, \chi_1 (\xi,\mu)\, d\xi.\]
		As explained in \Cref{thm:zeroes2}, there are different scenarios depending on the number of solutions to $\tau- \Phi(\xi,\mu) = 0$ in $\mathcal{R}_1$. If there are none, the proof is simpler and we omit it. If there are two solutions, \Cref{thm:zeroes2} shows that one is positive and one is negative, so we can reduce it to the case of a single zero after replacing $\chi_1$ by two cut-offs $\chi_1^{+}$ and $\chi_1^{-}$ that equal $\chi_1$ when $\pm\xi>0$ and zero otherwise. Therefore suppose there is a unique solution $\xi_0=\xi_0 (\tau,\mu)>0$ and consider a cut-off $\varphi_0 (\xi)$ as given by \Cref{thm:zeroes} which is supported in a ball $B(\xi_0,\varepsilon)$ with $\varepsilon = 10^{-6}\,|\xi_0|$.
		
		We define:
		\begin{align*}
			K_1 (\tau,x,\mu) & = \lim_{\delta\rightarrow 0}\, \int_{\delta<|\tau- \Phi(\xi,\mu)|<\delta^{-1}} e^{ix\xi}\, \frac{\xi^2}{\tau- \Phi(\xi,\mu)} \, \varphi_0(\xi)\, \chi_1 (\xi,\mu)\, d\xi\\
			& + \lim_{\delta\rightarrow 0}\, \int_{\delta<|\tau- \Phi(\xi,\mu)|<\delta^{-1}} e^{ix\xi}\, \frac{\xi^2}{\tau- \Phi(\xi,\mu)} \, [1-\varphi_0(\xi)]\, \chi_1 (\xi,\mu)\, d\xi = I+ I\! I.
		\end{align*}
		
		{\bf Step 1.} Let us focus on $I$ first. For $\xi\in B(\xi_0, \varepsilon)$, \Cref{thm:zeroes2} gives some independent constant $c>0$ such that
		\begin{equation}\label{eq:nearzero}
			|\Phi (\xi,\mu)-\tau|=|\Phi (\xi,\mu)-\Phi (\xi_0,\mu)|\geq c \, |\xi-\xi_0|\, |\xi|^2 .
		\end{equation} 
		Then we change variables $\xi=\xi_0 z$ for $z\in B(1, 10^{-6})$
		\begin{align*}
		 I & = \int_{\R} \frac{e^{ix\xi}}{\xi-\xi_0}\,\frac{\xi^2 \, (\xi-\xi_0)}{\Phi(\xi_0,\mu)- \Phi(\xi,\mu)} \, \varphi_0(\xi)\, \chi_1 (\xi,\mu)\, d\xi\\
		 & = \int_{\R} \frac{e^{ix\xi_0 z}}{z-1}\, \frac{\xi_0^3\, z^2 (z-1)}{\Phi(\xi_0,\mu)- \Phi(\xi_0\,z,\mu)} \, \varphi_0(z\xi_0)\, \chi_1 (z\xi_0,\mu)\, dz \\
		 & =  i\pi\, e^{-ix\xi_0}\, \mbox{sign}(x) \ast_x \widehat{g}(x,\xi_0,\mu).
		 \end{align*}
		 where
		 \[ g(z,\xi_0,\mu) := \frac{\xi_0^3\, z^2 (z-1)}{\Phi(\xi_0,\mu)- \Phi(\xi_0\,z,\mu)} \, \varphi_0(z\xi_0)\, \chi_1 (z\xi_0,\mu).\]
		 
		By the Young convolution inequality, and integration by parts:
		\begin{equation*}
			\sup_{x\in\R} | \, I \, |  \lesssim \norm{ \widehat{g}}_{L^1_x}  = \norm{ \widehat{g}}_{L^1_x(|x|\leq 1)} + \norm{ \widehat{g}}_{L^1_x(|x|> 1)}\lesssim \norm{g}_{L^1_z}+ \norm{g''}_{L^1_z}.
		\end{equation*}
		We just need to guarantee that the last two norms admit a uniform bound in $\xi_0,\mu$. We do the first one as an example. By \eqref{eq:nearzero},
\begin{align*}
\norm{g}_{L^1_{z}} & \lesssim \int_{B(1,10^{-6})}  \frac{\xi_0^3\, z^2\, |z-1|}{|\xi_0\, z|^2 |\xi_0 z- \xi_0|} \, \varphi_0(z\xi_0)\, \chi_1 (z\xi_0,\mu) \, dz \lesssim  \int_{B(1,10^{-6})}  \, dz \lesssim 1.
\end{align*}
		
		{\bf Step 2.} Now we study $I\! I$. We subdivide this into three parts: $100<\xi<\frac{1}{2}\xi_0$, $\frac{1}{2}\xi_0<\xi<2\xi_0$ and $2\xi_0<\xi$. Consider a function $\varphi_j \in C_c^{\infty}$, $0\leq \varphi_j\leq 1$ supported at each of these regions for $j=1,2,3$, and such that they add up to one.
		
		We first consider the region where $100<\xi<\frac{1}{2}\xi_0$ (if there is no such region, then we have an upper bound for $\xi_0$ and the argument simplifies). Using \eqref{eq:nearzero} we have that
		\begin{multline*}
		 \Big | \int_{100<\xi<\frac{1}{2}\xi_0}\, e^{ix\xi}\, \frac{\xi^2}{\tau- \Phi(\xi,\mu)} \, [1-\varphi_0(\xi)]\, \chi_1 (\xi,\mu)\, \varphi_1(\xi)\, d\xi\Big | \\
	 \lesssim \int_{100<\xi<\frac{1}{2}\xi_0} \frac{1}{|\xi-\xi_0|} \, [1-\varphi_0(\xi)]\, \chi_1 (\xi,\mu)\, d\xi \lesssim \int_{100<\xi<\frac{1}{2}\xi_0}\frac{1}{\xi_0} d\xi \lesssim 1.
		\end{multline*}
		Now consider the region $\frac{1}{2}\xi_0 < \xi< 2\xi_0$. We use \eqref{eq:nearzero} once again:
		\begin{multline*}
		 \Big | \int_{\frac{1}{2}\xi_0 < \xi< 2\xi_0}\, e^{ix\xi}\, \frac{\xi^2}{\tau- \Phi(\xi,\mu)} \, [1-\varphi_0(\xi)]\, \chi_1 (\xi,\mu)\, \varphi_2(\xi)\, d\xi\Big | \\
		  \lesssim \int_{\frac{1}{2}\xi_0 < \xi< 2\xi_0} \frac{1}{|\xi-\xi_0|} \, [1-\varphi_0(\xi)]\, \chi_1 (\xi,\mu)\, d\xi \lesssim \int_{\frac{1}{2} < z< 2} \frac{1}{|z-1|} \,  | 1-\varphi_0(\xi_0\,z)|\, dz\\
		 \lesssim \int_{1/2}^{1-10^{-6}} \frac{1}{1-z}\, dz +\int_{1+10^{-6}}^{2} \frac{1}{z-1}\, dz  \lesssim 1.
		\end{multline*}
		
		Finally, we study $I\! I$ when $\xi>2\xi_0$, where we need to exploit cancellation to improve on the $|\xi|^{-1}$ decay. To do that, we separate the phase into top and lower order terms:
		\[ \Phi(\xi,\mu)= \frac{1}{16} \xi^3 - \frac{3}{8}\, \mu^2 \xi + \phi (\xi,\mu).\]
We write:
		\begin{equation}\label{eq:toporder2}
			\frac{\xi^2}{\Phi (\xi,\mu)-\Phi (\xi_0,\mu)} =\frac{1}{\xi\,\left(\frac{1}{16}-\frac{3\,\mu^2}{8\,\xi^2}\right)}+ \frac{\Phi(\xi_0,\mu)-\phi (\xi,\mu)}{\xi\,\left(\frac{1}{16}-\frac{3\,\mu^2}{8\,\xi^2}\right)\, (\Phi (\xi,\mu)-\Phi (\xi_0,\mu))}.
		\end{equation}
		Above, we have singled out the top order, given that the second term can be integrated in absolute value and produces a uniformly bounded contribution. Indeed, we use \eqref{eq:nearzero} as follows:
		\begin{align*}
			\int_{\xi>2\xi_0} \frac{|\Phi(\xi_0,\mu)-\phi (\xi,\mu)|}{|\xi|\,\left(\frac{1}{16}-\frac{3\,\mu^2}{8\,\xi^2}\right)\, |\Phi (\xi,\mu)-\Phi (\xi_0,\mu)|}\, [1-\varphi_0(\xi)]\, \chi_1 (\xi,\mu)\, d\xi & \lesssim \int_{\xi>2\xi_0} \frac{|\xi_0|^3 + |\xi|^2}{|\xi|^3 \, |\xi-\xi_0|}\, d\xi\\
			& \lesssim \int_{\xi>2\xi_0} \frac{|\xi_0|^3 + |\xi|^2}{|\xi|^4}\, d\xi \lesssim 1.
		\end{align*}
		Therefore, we need only control the top order in \eqref{eq:toporder2}. We rewrite it as:
\[ \frac{\varphi_3 (\xi)\, \chi_1 (\xi,\mu)}{\xi\,\left(\frac{1}{16}-\frac{3\,\mu^2}{8\,\xi^2}\right)}= \frac{16\, \varphi_3 (\xi)\, \chi_1 (\xi,\mu)}{\xi} + 16\,\varphi_3 (\xi)\, \chi_1 (\xi,\mu)\, \frac{6\mu^2 }{\xi\, (\xi^2-6\,\mu^2)}. \]
The latter is easy to control after the change of variables $\xi=\mu z$,
\[ \int_{\R} e^{ix\xi} \varphi_3 (\xi)\, \chi_1 (\xi,\mu)\, \frac{\mu^2 }{\xi\, (\xi^2-6\,\mu^2)}\, d\xi = \int_{\R} e^{ix\mu z} \varphi_3 (\mu z)\, \chi_1 (\mu z,\mu)\, \frac{1}{z\, (z^2-6)}\, dz.\]
Note that the integrand is supported in the region $\{ z> \max\{ 200, \frac{2\xi_0}{|\mu|}\} \}$ and is absolutely integrable.

{\bf Step 3.} Finally, we estimate the term
\[ \int_{\R} e^{ix\xi} \, \frac{\varphi_3 (\xi)\, \chi_1 (\xi,\mu)}{\xi} \, d\xi.\]
We want to show that this is uniformly bounded in $\mu$ and $\xi_0$ (since $\varphi_3$ depends on it). The idea is to rescale the variable $\xi$, but there are two cases to consider:
\begin{enumerate}
\item $100\, |\mu| > \xi_0(\mu)$, and 
\item $100\, |\mu| \leq  \xi_0(\mu)$.
\end{enumerate}

In case $(1)$, we do the change of variables $\xi=\mu z$:
\[ \int_{\R} e^{ix\xi} \, \frac{\varphi_3 (\xi)\, \chi_1 (\xi,\mu)}{\xi} \, d\xi =\int_{\R} e^{ix\mu z} \, \frac{\varphi_3 (\mu z)\, \chi_1 (\mu z,\mu)}{z} \, dz.\]
Now note that the integrand is supported in the region given by $z>\max\{ 2\frac{|\xi_0 |}{\mu}, 200 \}=200$. Therefore, 
\begin{align*}
\int_{\R} e^{ix\mu z} \, \frac{\varphi_3 (\mu z)\, \chi_1 (\mu z,\mu)}{z} \, dz & = \int_{\R} e^{ix\mu z} \frac{dz}{z} - \int_{\R} e^{ix\mu z} \frac{1- \varphi_3 (\mu z)\, \chi_1 (\mu z,\mu)}{z}\\
& = -i\pi\, e^{ix\mu}\,\mbox{sign}(x) + i \pi \, [e^{i\mu\,\cdot}\, \mbox{sign}(\cdot) \ast [1- \varphi_3 (\mu \cdot )\, \chi_1 (\mu \cdot,\mu)]^{\wedge}](x)
\end{align*}
It is easy to show that the last term is bounded thanks to the fact that $1- \varphi_3 (\mu \cdot )\, \chi_1 (\mu \cdot,\mu)$ is a smooth function supported in $|z|\leq 200$.

In case $(2)$, we do the change of variables $\xi=\xi_0 z$ and run a similar argument to the above to obtain a uniformly bounded contribution.
	\end{proof}
	
	Thanks to \Cref{thm:3}, the following is straight-forward:
	
	\begin{thm}[Smoothing effect]\label{thm:smoothingeffect}
		For $f\in C_c^{\infty}(\R^3)$, we have that
		\begin{align*}
			\norm{\pa_x^2 \widetilde{W_1}(t) Q_j f}_{L^{\infty}_{x} L^2_{y,t}} & \lesssim \norm{Q_j f}_{L^1_{x} L^2_{y,t}},\\
			\norm{\pa_x\, \pa_y\, \widetilde{W_2}(t)Q_j f}_{L^{\infty}_{y} L^2_{x,t}} & \lesssim \norm{Q_j f}_{L^1_{y} L^2_{x,t}}.
		\end{align*}
	\end{thm}
	\begin{proof} We prove the first estimate. By \Cref{thm:1}, \Cref{thm:2} and \Cref{thm:3} we have that:
		\begin{align*}
			\norm{\pa_x^2 \widetilde{W_1}(t)Q_j f}_{L^{\infty}_{x} L^2_{y,t}} & = \norm{ \int_{\mu,\tau} e^{iy\mu+it\tau} K_1(\tau,\cdot,\mu)\ast_{x} \widehat{Q_j f}^{(y,t)}(\tau,\cdot,\mu) \, d\tau\, d\mu}_{L^{\infty}_{x} L^2_{y,t}}\\
			& \lesssim \norm{ K_1(\tau,\cdot,\mu)\ast_{x} \widehat{Q_j f}^{(y,t)}(\tau,\cdot,\mu)}_{L^{\infty}_{x} L^2_{\mu,\tau}}\\
			& \lesssim \sup_{x} \int_{\R}\left( \int_{\R^2}  |K_1(\tau,x-\tilde{x},\mu)|^2\, |\widehat{Q_j f}^{(y,t)}(\tau,\tilde{x},\mu)|^2 \, d\tau d\mu\right)^{1/2} \, d\tilde{x} \\
			& \lesssim \int_{\R} \left(\int_{\R^2} |\widehat{Q_j f}^{(y,t)}(\tau,\tilde{x},\mu)|^2 \,d\tau d\mu \right)^{1/2} d\tilde{x}= \norm{Q_j f}_{L^1_{x} L^2_{y,t}}.
		\end{align*}
		We have first used the Plancherel inequality, and then the Minkowski inequality together with the fact that $K_1$ is uniformly bounded (\Cref{thm:3}).
	\end{proof}
	
	By \eqref{eq:relationshipW}, this theorem implies the smoothing effect for the Duhamel term. 
	\begin{cor}[Double smoothing effect]  The following estimates hold:
		\begin{align*}
			\norm{\pa_x^2 \int_{0}^t W_1 (t-t')Q_j F( t') \, dt'}_{L^{\infty}_{x}L^2_{y,t}} & \lesssim \norm{Q_j F}_{L^{1}_{x}L^2_{y,t}},\\
			\norm{\pa_x \, \pa_y \, \int_{0}^t W_2 (t-t')Q_j F( t') \, dt'}_{L^{\infty}_{y}L^2_{x,t}} & \lesssim \norm{Q_j F}_{L^{1}_{y}L^2_{x,t}}.
		\end{align*}
	\end{cor}
	
\section{Maximal function estimates}\label{sec:maxfunc}
	
	\subsection{$TT^{\ast}$ argument}\label{sec:TTarg}
	
	We will follow the techniques in \cite{KPV} and also \cite{KZ}, which are useful thanks to the anisotropic nature of the KP equation. Let us start by defining:
	\begin{align}
		I_1^s (t,x,y) &  := \int_{\R^2} e^{i\xi x + i \mu y + i t \Phi (\xi,\mu)} \,  |\xi|^{-s} \, \chi_1 (\xi,\mu)\, d\mu \, d\xi,\\
		I_2^s (t,x,y) & := \int_{\R^2} e^{i\xi x + i \mu y + i t \Phi (\xi,\mu)} \,  |\mu|^{-s} \, \chi_2 (\xi,\mu)\, d\mu \, d\xi,
	\end{align}
	where $s>0$ will be made explicit later. We are looking for an estimate such as:
	\begin{equation}\label{eq:firstguess}
		\norm{W_1 (t) f}_{L^4_x L^{\infty}_{t,y}} \lesssim \norm{D_x^s \, f}_{L^2_{x,y}}.
	\end{equation}
	
	First we write the dual to estimate \eqref{eq:firstguess}:
	\begin{equation}\label{eq:dualest}
		\norm{\int_{\R} W_1(t) g(t,\cdot,\cdot)\, dt}_{L^2_{x,y}}\lesssim \norm{D_x^s \,g}_{L^{4/3}_x L^1_{t,y}}.
	\end{equation}
	By a $TT^{\ast}$ argument, the LHS of \eqref{eq:dualest} can be rewritten as:
	\[ \norm{\int_{\R} D_x^{-s} W_1(t) g(t,\cdot,\cdot)\, dt}_{L^2_{x,y}}^2 \lesssim \norm{g}_{L^{4/3}_x L^1_{t,y}}\, \norm{\int_{\R} D_x^{-2s}W_1 (t-t') \overline{g(t',x,y)} \, dt'}_{L^4_x L^{\infty}_{t,y}}.\]
	This means that \eqref{eq:dualest} is equivalent to:
	\begin{equation}\label{eq:maxest}
		\norm{\int_{\R} D_x^{-2s}\, W_1 (t-t') g(t',x,y) \, dt'}_{L^4_x L^{\infty}_{t,y}} \lesssim \norm{g}_{L^{4/3}_x L^1_{t,y}}.
	\end{equation}
	Note that 
	\[ \int_{\R} D_x^{-2s}\, W_1 (t-t') g(t',x,y) \, dt' = I_1^{2s} \ast g,\]
	where the convolution is in all three variables $t,x,y$. Now we may use the Hardy-Littlewood-Sobolev inequality to place this in the desired space.
	\[ \norm{I_1^{2s} \ast g}_{L^4_x L^{\infty}_{T,y}} \lesssim  \norm{ |\cdot|^{-\alpha} \ast_x  \norm{g}_{L^1_{T,y}} }_{L^4_x}  \lesssim\norm{g}_{L^{4/3}_x L^1_{T,y}}.\]
	as long as $\alpha=\frac{1}{2}$, where $\alpha=\alpha(s)$ in some way which will yield $s$.
	
	Therefore, the goal is to obtain an estimate of the form:
	\[ | I_1^{2s} (t,x,y)| \leq C \, |x|^{-1/2}\]
	for some $C$ independent of $t,y$.
		
	Instead of working with $I^{2s}_1$, let us define the following quantities
	\begin{equation}
	\label{eq:defikj}
	 I_{k,j} (t,x,y)= \int_{\R^2} e^{i\xi x + i \mu y + i t \Phi (\xi,\mu)} \, \alpha_k(\xi)\, \alpha_j (\mu)\, d\mu \, d\xi,
	 \end{equation}
	where $k,j\geq 0$. Here $\alpha_n$ is a $C_c^{\infty}(\R)$ function with $\mbox{supp}(\alpha_n)\subset [2^{n-1}, 2^{n+1}]$ (and $[0,2]$ in the case of $n=1$). We will implicitly consider the case $\xi,\mu,t>0$ but an analogous argument allows other possibilities. Note that direct integration yields the trivial estimate
	\begin{equation}\label{eq:trivialbound}
		| I_{k,j}(t,x,y)|\lesssim 2^{j+k}.
	\end{equation}
	We start our analysis by studying the $\mu$-integral:
	\[ f(\xi; t,y):= \int_{\R}  e^{i\mu y+ i t \, \left( \frac{1}{4}-\frac{3}{8}\xi\right) \, \mu^2} \, \alpha_j(\mu)\, d\mu.\]
	The phase has a stationary point at
	\begin{equation}\label{eq:defmu0}
	 \mu_0 = \frac{-y}{ t \, \left( \frac{1}{4}-\frac{3}{8}\xi\right)}.
	 \end{equation}
We will distinguish various cases in our analysis, depending on whether we are in the case of a stationary point (whenever $|\mu_0|\sim 2^j$) or not. We will also assume that $k\geq 7$ and $j<k-4$, which corresponds to the frequency-region $\mathcal{R}_1$, and will later consider other possibilities.
	
	\subsection{There is a stationary point}\label{sec:stationarypoint}
	
	We will first assume that $\mu_0$ is in the support of $\alpha_i$ for some $i\in \{j-2,j-1,j,j+1,j+2\}$. From \eqref{eq:defmu0}, it follows that $|y| \sim t\, 2^{k+j}$. Then we can use the following lemma to obtain asymptotics for our oscillatory integral. There are many versions of this result, here we follow Lemma 2.5 in \cite{KZ} (note that we added $e^{i\lambda \phi(x_0)}$ as it seems to be missing):
	
	\begin{lem}\label{thm:oscillatoryasymp} Suppose that 
		\[ \phi(x_0)=0,\ \phi'(x_0)=0,\ \mbox{and}\ \phi''(x_0)\neq 0.\]
		Suppose that $\psi$ is smooth and supported in a neighborhood of $x_0$ which is sufficiently small and contains at most one critical point of $\phi$. Then
		\[ I(\lambda)=\int e^{i\lambda \phi(x)} \, \psi(x)\,dx \approx \lambda^{-1/2}\,e^{i\lambda \phi(x_0)} \, \sum_{j=0}^{\infty} a_j \lambda^{-j/2},\]
		in the sense that for all $N,r\geq 0$
		\[ \left(\frac{d}{d\lambda}\right)^r \left[ I(\lambda) -  \lambda^{-1/2}\, e^{i\lambda \phi(x_0)}\, \sum_{j=0}^{N} a_j \lambda^{-j/2} \right] = \O (\lambda^{-r-(N+2)/2})\]
		as $\lambda\rightarrow\infty$.
		
		Moreover the bounds in the error term depend on upper bounds of finitely many derivatives of $\phi$ and $\psi$ in the support of $\psi$, the size of the support of $\psi$, and a lower bound for $|\phi''(x_0)|$. Furthermore, $a_j=0$ for $j$ odd and 
		\[ a_0 = \left(\frac{c}{-i\phi''(x_0)}\right)^{1/2}\, \psi(x_0) .\]
	\end{lem}
	
	It is fundamental that we make sure that the implicit constants from using this result do not depend on any of our variables. In order to do that, we first change variables in the integral defining $f$ and take $\lambda= y\,\mu_0$, so that 
\[ f(\xi; t,y) = \int_{\R}  e^{i\, \lambda \, \left( \frac{\mu}{\mu_0}- \frac{\mu^2}{\mu_0^2}\right)} \, \alpha_j(\mu)\, d\mu= \mu_0\, \int_{\R}  e^{i\, \lambda \, \left( \tilde{\mu}- \tilde{\mu}^2\right)} \, \alpha_j(\tilde{\mu} \mu_0)\, d\tilde{\mu}.
\]
	Note that $\mu_0$ has size $2^j$, so intuitively $\alpha_j(\tilde{\mu} \mu_0)\approx \alpha_1 (\tilde{\mu})$. This admits uniform bounds and uniform support on $|\tilde{\mu}|\sim 1$, thus the new phase $\phi(\tilde{\mu})=\tilde{\mu}- \tilde{\mu}^2$ also admits uniform bounds. Consequently, we can use \Cref{thm:oscillatoryasymp} with 
	\[ I (\lambda)= \int_{\R}  e^{i\, \lambda \, \left( \tilde{\mu}- \tilde{\mu}^2\right)} \, \alpha_j(\tilde{\mu} \mu_0)\, d\tilde{\mu}.\]
	In particular we have that
	\[ I(\lambda)= a_0 \, \lambda^{-1/2}\, e^{i\lambda \phi(x_0)}+ E(\lambda)\]
	with uniform bounds:
\[ |E(\lambda)| \lesssim \lambda^{-1},\qquad |E'(\lambda)| \lesssim \lambda^{-2}.\]
	Note that this asymptotic expansion is relevant as long as $\lambda>1$. In this case, $|\lambda|\sim |y\, \mu_0|\sim t \,2^{2j+k}$.
	Therefore we are working in the regime where
	\begin{equation}\label{eq:timeregime}
		t\gtrsim 2^{-2j-k}.
	\end{equation}
	Going back to the variable $\xi$, we use that $\frac{d\lambda}{d\xi}= \frac{\lambda}{\xi}$ to find:
	\begin{equation}\label{eq:fasymp}
		f(\xi; t,y)  =   \frac{a_0}{ t^{1/2} \, \left( \frac{1}{4}-\frac{3}{8}\xi\right)^{1/2}} \, e^{\frac{i}{2}\,y\mu_0}+  \frac{y}{ t \, \left( \frac{1}{4}-\frac{3}{8}\xi\right)} \, E(\lambda),
		\end{equation}
	with the following error bounds:
	\begin{align}\label{eq:errorasymp}
		\Big | \frac{y}{ t \, \left( \frac{1}{4}-\frac{3}{8}\xi\right)} \, E(\lambda)\Big| & \lesssim 2^{j} |\lambda|^{-1} = t^{-1}\, 2^{-j-k},\\
		\Big | \frac{d}{d\xi}\left(  \frac{y}{ t \, \left( \frac{1}{4}-\frac{3}{8}\xi\right)} \, E(\lambda)\right) \Big| & \lesssim 2^{j-k} |\lambda|^{-1} = 2^{-k} \, |y|^{-1}= 2^{-k}\, (t 2^{k+j})^{-1} = 2^{-j-2k}\, t^{-1}.\nonumber
	\end{align}
	Using \eqref{eq:fasymp} we may write:
	\begin{align*}
		I_{k,j} (t,x,y) & = a_0 \, \int_{\R} e^{ix\xi + i t\, P(\xi)+\frac{i}{2} y\mu_0} \, \alpha_k (\xi) \,   \frac{1}{ t^{1/2} \, \left( \frac{1}{4}-\frac{3}{8}\xi\right)^{1/2}} \, d\xi + \E_{k,j} (t,x,y)\\
		&  = a_0 \, \mathcal{I}_{k,j}(t,x,y) + \E_{k,j} (t,x,y).
	\end{align*}
	where
		\[ P(\xi)=\frac{1}{16}\xi^3 - \frac{1}{8}\xi^2 +\frac{1}{2}\xi.\]
	 and the error term is
	\[ \E_{k,j} (t,x,y)= \int_{\R} e^{ix\xi + i t\, P(\xi)+\frac{i}{2} y\mu_0} \, \alpha_k (\xi) \,  \frac{y}{ t \, \left( \frac{1}{4}-\frac{3}{8}\xi\right)} \, E(y\mu_0)\,d\xi.\]
	
	We will first study the integral given by the top order:
	\begin{equation}\label{eq:toporder5}
		\mathcal{I}_{k,j} (t,x,y)  =\int_{\R} e^{ix\xi + i t\, P(\xi)+\frac{i}{2} y\mu_0} \, \alpha_k (\xi) \,   \frac{1}{ t^{1/2} \, \left( \frac{1}{4}-\frac{3}{8}\xi\right)^{1/2}} \, d\xi.
	\end{equation}
	Recall that $|y|\sim t \, 2^{k+j}$, and that $\mu_0$ was defined in \eqref{eq:defmu0}. We compute the derivative of the phase of $\mathcal{I}_{k,j}$:
	\begin{equation*}
 \phi(\xi)  := x\xi + t\, P(\xi)-\frac{y^2}{ 2 t \, \left( \frac{1}{4}-\frac{3}{8}\xi\right)},\qquad \phi'(\xi)  = x + t\, P'(\xi) +\frac{3}{8}\, \frac{y^2}{ t \, \left( \frac{1}{4}-\frac{3}{8}\xi\right)^2}.
\end{equation*}
	The term $ t\, P'(\xi)$ has size $t 2^{2k}$, while the term in $y$ only has size $t\, 2^{2j}$ (recall that we are in the case $j< k-4$). There are two cases to consider:
	\begin{itemize}
		\item When $|x|\gtrsim t 2^{2k}$, we have that $|\phi'(\xi)|\gtrsim |x|$, so \Cref{thm:VanderCorput} yields:
		\[ |\mathcal{I}_{k,j}(t,x,y)|\lesssim |x|^{-1} \, (\norm{\psi}_{L^{\infty}_{\xi}} + \norm{\pa_{\xi} \psi}_{L^{1}_{\xi}})\]
		where 
		\[ \psi (\xi)= \alpha_k (\xi) \,   \frac{1}{ t^{1/2} \, \left( \frac{1}{4}-\frac{3}{8}\xi\right)^{1/2}}.\]
		Note that $\norm{\psi}_{L^{\infty}_{\xi}}, \norm{\pa_{\xi} \psi}_{L^{1}_{\xi}} \lesssim t^{-1/2} \, 2^{-k/2}$ and therefore:
		\begin{equation}\label{eq:max1}
			|\mathcal{I}_{k,j}(t,x,y)|\lesssim |x|^{-1} \, 2^{-k/2} \, t^{-1/2} \qquad \mbox{when}\ |x|\gtrsim t 2^{2k}.
		\end{equation}
		
		\item When $|x|\lesssim t 2^{2k}$  we use the second derivative of the phase:
		\[ \phi''(\xi)=t\, P''(\xi)-\frac{27}{64}\,\frac{y^2}{ t \, \left( \frac{1}{4}-\frac{3}{8}\xi\right)^3}\]
		and has size $t \, 2^{k}$ from the first term. In that case, \Cref{thm:VanderCorput} gives the bound:
		\begin{equation}\label{eq:max2}
			|\mathcal{I}_{k,j}(t,x,y)|\lesssim |t\, 2^{k}|^{-1/2} \, t^{-1/2} 2^{-k/2} \lesssim t^{-1/2} \, |x|^{-1/2}\qquad \mbox{when}\ |x|\lesssim t 2^{2k}.
		\end{equation}
	\end{itemize}
	
	We can finally present our findings:
	\begin{prop} Suppose that $t>2^{-2j-k}$,  $k\geq 7$, $j\geq 0$ and $j<k-4$. Then we have that
		\[ |\mathcal{I}_{k,j} (t,x,y)|\lesssim |x|^{-1/2} \, 2^{j+k/2}.\]
		where the implicit constant is independent of $t,x,y,k,j$.
	\end{prop}
	\begin{proof}
		Let us start with the case $|x|\gtrsim t 2^{2k}$. We interpolate between \eqref{eq:max1} and \eqref{eq:trivialbound} (which also holds for $\mathcal{I}_{k,j}$), and then use that $t>2^{-2j-k}$:
		\[ |\mathcal{I}_{k,j}(t,x,y)|\lesssim |x|^{-1/2} \, t^{-1/4} 2^{j/2 + k/4} \lesssim |x|^{-1/2} \, 2^{j+k/2}.\]
		
		In the case $|x|\lesssim t \, 2^{2k}$, we directly use \eqref{eq:max2} together with  $t>2^{-2j-k}$:
		\[ |\mathcal{I}_{k,j}(t,x,y)|\lesssim |x|^{-1/2}\, t^{-1/2}\lesssim |x|^{-1/2} \, 2^{j+k/2}.\]
	\end{proof}
	\begin{rk} Using \eqref{eq:errorasymp}, it is easy to check that $\E_{k,j}$ satisfies exactly the same estimates, and therefore the result extends from the top order, $\mathcal{I}_{k,j}$, to the full kernel $I_{k,j}$.
	\end{rk}
	When $t\leq 2^{-2j-k}$ we can directly use \eqref{eq:trivialbound} to obtain the following
	\begin{prop}\label{thm:smalltimemaximal} Suppose that $t\leq 2^{-2j-k}$,  $k\geq 7$, $j\geq 0$ and $j<k-4$. Then we have that
		\[ |I_{k,j} (t,x,y)|\lesssim |x|^{-1/2} \, 2^{j+k/2}.\]
		where the implicit constant is independent of $t,x,y,k,j$.
	\end{prop}
	\begin{proof}
		Integrationg by parts gives:
		\[ I_{k,j}(t,x,y)= x^{-1} \int_{\R} e^{ix\xi} \, \frac{d}{d\xi}\left( \alpha_k(\xi)\, f(t,x,\xi)\right)\, d\xi.\]
		Then we may use the trivial bound $|\pa_{\xi} f(t,x,\xi)|\lesssim t \, 2^{3j}$ and $|f(t,x,\xi)|\lesssim 2^{j}$ to obtain the bound
		\[ |I_{k,j}(t,x,y)|\lesssim |x|^{-1} \, \max\{ t\, 2^{3j+k},  2^{j}\}.\]
		Now we use the fact that $t\leq 2^{-2j-k}$ to obtain:
		\[ |I_{k,j}(t,x,y)|\lesssim |x|^{-1} \, 2^{j}.\]
		Finally, we interpolate between this and \eqref{eq:trivialbound} to obtain the estimate $|x|^{-1/2} 2^{j+k/2}$.
	\end{proof}
	
	\begin{rk} These results have been obtained under the assumption that we have a stationary point $\mu_0$. This is the most difficult case, and therefore these results extend to the case where there is no stationary point. The latter is briefly discussed in \Cref{sec:nostationary}.
	\end{rk}
	
As explained in \Cref{sec:TTarg}, our final loss of derivatives, $s$, is chosen to offset the growth of $2^{j+k/2}$, i.e.
\[ (-2s)\cdot k + j+\frac{k}{2}<0.\]
This yields a loss of derivatives of $s>3/4$. We summarize our findings in the following theorem:
	\begin{thm}\label{thm:maximal} Suppose that  $k\geq 7$, $j\geq 0$ and $j<k-4$. Then we have that 
		\[ |I_{k,j} (t,x,y)|\lesssim |x|^{-1/2} \, 2^{j+k/2}.\]
		uniformly in $t$ and $y$. By the $TT^{\ast}$ argument in \Cref{sec:TTarg}, this implies the following maximal function estimate for any $s>3/4$:
		\[ \norm{W_1 (t) u_0}_{L^4_x L^{\infty}_{t,y}} \lesssim \norm{\langle\nabla\rangle^s \, u_0}_{L^2_{x,y}}.\]
	\end{thm}
	\begin{rk}
	One may compare these estimates with the maximal function estimates in the work of Linares and Pastor \cite{LinaresPastor} for the modified Zakharov-Kuznetsov equation. In fact the range for $s$ is the same, despite the different dispersive relation.
	\end{rk}
	
	As a direct consequence, we obtain a similar estimate for the Duhamel term:
	
	\begin{cor}\label{thm:maximal2} For $s>3/4$, we have 
		\[ \norm{ \int_0^t W_1 (t-t') F(t')\, dt'}_{L^4_x L^{\infty}_{T,y}} \lesssim \norm{\langle\nabla\rangle^s F}_{L^1_T L^2_{x,y}}.\]
	\end{cor}
	\begin{proof}
		Let $\chi_{[0,t]}(t')$ be a time cut-off and $\tilde{F}(t'):= W_1 (-t') F(t')$. By the Minkowski inequality,
		\begin{align*}
			\norm{ \int_0^t W_1 (t-t') F(t')\, dt'}_{L^4_x L^{\infty}_{t,y}} & \lesssim \int_0^T \norm{ \chi_{[0,t]} W_1 (t-t') F(t')}_{L^4_x L^{\infty}_{T,y}} \, dt'\\
			& \lesssim \int_0^T \norm{W_1 (t) \tilde{F}(t')}_{L^4_x L^{\infty}_{T,y}} \\
			& \lesssim \int_0^T \norm{\langle\nabla\rangle^s \tilde{F}(t')}_{L^2_{x,y}} \, dt'=\norm{\langle\nabla\rangle^s F}_{L^1_T L^2_{x,y}}.
		\end{align*}
		The last step follows from the fact that $W_1(-t')$ is unitary in $L^2_{x,y}$.
	\end{proof}
	
	\subsection{There is no stationary point}\label{sec:nostationary}
	
Suppose that the stationary point $\mu_0$ is not in the support of  $\alpha_i$ for any $i\in\{j-2,j-1,j,j+1,j+2\}$. In that case $\mu$ and $\mu_0$ have different orders of magnitude, in particular $|\mu|>4 |\mu_0|$ or $|\mu_0|>4|\mu|$. We rewrite
	\begin{equation}\label{eq:muintegral}
	 f(t,\xi,y) =\int_{\R} e^{it(\frac{1}{4}-\frac{3}{8}\xi)\, \mu\, (\mu-\mu_0)} \, \alpha_j (\mu)\, d\mu.
	 \end{equation}
	Suppose for instance that $\mu>4\, \mu_0$. In the first case, the derivative of the phase is bounded below by $t\, 2^{k+j}$ and therefore \Cref{thm:VanderCorput} yields:
	\begin{equation*}
	|f(t,\xi,y)|  \lesssim \min\{ 2^{j},\  t^{-1} \, 2^{-k-j} \},\qquad |\pa_{\xi} f(t,\xi,y)| \lesssim \min\{ t\, 2^{3j},\ 2^{j-k}\}.
	\end{equation*}
	which are better than before. Then one uses \Cref{thm:VanderCorput} again on the $\xi$-integral to obtain better control than when we had a stationary point. 
	
	\subsection{The case of $W_2$}
	
	We will also need to estimate $W_2$ in $L^4_x L^{\infty}_{T,y}$. This can be achieved by studying $I_{k,j}$ in \eqref{eq:defikj}, but in the case $j\geq k-4$ and $k\geq 7$. The arguments presented above apply to this case with minimal changes: \eqref{eq:trivialbound} is still true, as are \eqref{eq:timeregime}, \eqref{eq:fasymp} and \eqref{eq:errorasymp}. \eqref{eq:max1} remains true after a small change:
	\[ |\mathcal{I}_{k,j}(t,x,y)|\lesssim |x|^{-1} \, 2^{-k/2} \, t^{-1/2} \qquad \mbox{when}\ |x|\gtrsim t 2^{2j}.\]
	Now we have $|x|\gtrsim t \, 2^{2j}$ instead of $t\, 2^{2k}$ since $j$ dominates $k$.
	
	Similarly, \eqref{eq:max2} is still valid in the regime where $|x|\lesssim t \, 2^{2j}$. Regarding \eqref{eq:max2}, it is important to note that there is no cancellation of the second derivative; note that
	\[ \phi''(\xi) =t\, P''(\xi)-\frac{27}{64}\,\frac{y^2}{ t \, \left( \frac{1}{4}-\frac{3}{8}\xi\right)^3}= t\, \left(\frac{3}{8}\xi - \frac{1}{4}\right) + \frac{27}{64}\,\frac{y^2}{ t \, \left( \frac{3}{8}\xi - \frac{1}{4}\right)^3}.\]
	Therefore, when $\xi>100$ both summands have the same sign, and the same happens when $\xi<-100$.
	
	All in all,  we may bound $|\phi''(\xi)|\gtrsim t \, 2^{j}$ whenever $|x|\lesssim t \, 2^{2j}$, thus recovering \eqref{eq:max2}. We then obtain
	\begin{align*}
	 |\mathcal{I}_{k,j} (t,x,y)|& \lesssim |x|^{-1/2} \, 2^{j+k/2} \qquad \mbox{when}\ t>2^{-2j-k},\\
 |\mathcal{I}_{k,j} (t,x,y)| & \lesssim |x|^{-1/2} \, 2^{3k/2} \qquad \mbox{when}\ t\leq 2^{-2j-k}.
 \end{align*}
	This time, the first estimate dominates given that $j\geq k-4$. However, this makes no difference to the loss of derivatives, and we again have
	\[ \norm{W_2 (t) u_0}_{L^4_x L^{\infty}_{T,y}}\lesssim \norm{\langle\nabla\rangle^s u_0}_{L^2_{x,y}} \]
	for any $s>3/4$.
	
	\subsection{Maximal function estimates in $L^4_y L^{\infty}_{T,x}$}
	
	Our goal in this subsection is to derive maximal function estimates for $I_{k,j}$ in \eqref{eq:defikj} in the space $L^4_y L^{\infty}_{T,x}$. The $TT^{\ast}$ argument from \Cref{sec:TTarg} shows that it is enough to obtain an estimate of the form:
	\[ |I_{k,j} (t,x,y)|\leq C\, |y|^{-1/2},\]
	where $C$ might depend on $j,k$ (which will yield $s$), but not $t,x$. 
	
	In order to do this, one can repeat the arguments above while inverting the order of the $\mu$ and $\xi$ integrals. Alternatively, we can use a shortcut. Let us give an example in the more difficult case, where we have a stationary point. From \eqref{eq:defmu0}, we know that $|y|\sim t \, 2^{j+k}$ which we can use to trade decay in $|x|$ for decay in $|y|$. As an example, consider the case where $k\geq 7$, $j\geq 0$, $j<k-4$ and  $t>2^{-2j-k}$. We distinguish two cases:
	\begin{itemize}
	\item When $|x|\gtrsim t \, 2^{2k}$, we have  that $|x|\gtrsim 2^{k-j}\, |y|$. We use this bound on \eqref{eq:max1} together with \eqref{eq:trivialbound} to obtain:
	\[ |\mathcal{I}_{k,j} (t,x,y)|\lesssim  |y|^{-1/2}\, 2^{3j/2}.\]
	\item  When $|x|\lesssim t \, 2^{2k}$, we directly use $|y|\sim t \, 2^{j+k}$ on \eqref{eq:max2}, together with \eqref{eq:trivialbound} to obtain:
	\[ |\mathcal{I}_{k,j} (t,x,y)|\lesssim |y|^{-1/2} \, 2^{j+k/2}.\]
	\end{itemize}
	
	Analogous estimates take care of the remaining cases, which we summarize below.
	\begin{thm}\label{thm:maximal3} Suppose that  $k\geq 7$ and $j\geq 0$. Then we have that 
		\[ |I_{k,j} (t,x,y)|\lesssim |y|^{-1/2} \, \max\{ 2^{3k/2}, 2^{3j/2}\}. \]
		uniformly in $t$ and $x$. 
		By the $TT^{\ast}$ argument in \Cref{sec:TTarg}, this implies the following maximal function estimate for any $s>3/4$ and any $r=1,2$:
		\[ \norm{W_r (t) u_0}_{L^4_y L^{\infty}_{t,x}} \lesssim \norm{\langle\nabla\rangle^s \, u_0}_{L^2_{x,y}}.\]
	\end{thm}
	
	\begin{cor}\label{thm:maximal4} For $s>3/4$ and any $r=1,2$ we have that
		\[ \norm{\int_0^t W_r (t-t') F(t') \, dt'}_{L^4_y L^{\infty}_{T,x}} \lesssim \norm{\langle\nabla\rangle^s F}_{L^1_T L^2_{x,y}}.\]
	\end{cor}
	
	
	\subsection{The case of $W_0$}
	
	For low frequencies, maximal function estimates can be derived using similar techniques. The corresponding kernel is
	\[ I_{j} (t,x,y) = \int_{\R^2} e^{ix\xi + i y\mu + i t\Phi (\xi,\mu)}\, \alpha_j(\mu)\, \chi_0(\xi)\, d\xi\, d\mu,\]
	where $\chi_0$ was defined in \eqref{eq:cut-offs}.
	Clearly, we have the trivial bound $|I_j(t,x,y)|\lesssim 2^{j}$ after direct integration. Let us explain how to derive estimates in the space $L^4_x L^{\infty}_{T,y}$ as an example, since the case of $L^4_y L^{\infty}_{T,x}$ is analogous.
	
	Following the proof of \Cref{thm:smalltimemaximal}, one may easily use trivial bounds to obtain
	\begin{equation}\label{eq:shortcut}
	 |I_j (t,x,y)|\lesssim |x|^{-1}\, \max\{ t\, 2^{3j},  2^{j}\}.
	\end{equation}
	After interpolating these with the bound $2^{j}$ we obtain
	\[ |I_j (t,x,y)|\lesssim |x|^{-1/2}\, \max\{ t^{1/2}\, 2^{2j},  2^{j}\}.\]
	When $t\leq 2^{-j}$, we directly have the desired bound $|x|^{-1/2} \, 2^{3j/2}$, so we focus on the case $t\geq 2^{-j}$ from now on. We can further restrict ourselves to the case when $|x|\gtrsim t\, 2^{3j}$, given that the opposite situation is dealt with using \eqref{eq:shortcut} directly.
	
 As before, the case of a stationary point is the more complicated one, so we start there. The stationary point is given by \eqref{eq:defmu0}. The analysis carried out in \Cref{sec:stationarypoint} still holds, and thus we have that the top order of $I_{j}$ is given by
		\[ \mathcal{I}_{j} (t,x,y) = \int_{\R} e^{i\phi(\xi)} \, \chi_0 (\xi) \,   \frac{1}{ t^{1/2} \, \left( \frac{1}{4}-\frac{3}{8}\xi\right)^{1/2}} \, d\xi,\]
	where the phase is 
	\[ \phi(\xi)= x\xi + t\, P(\xi)-\frac{y^2}{ 2 t \, \left( \frac{1}{4}-\frac{3}{8}\xi\right)}.\]
	We split this integral into two, given by the regions
	\begin{align*}
	\Omega_1 &  := \{ \xi\in \R \mid |\xi|<100,\ |\xi-\frac{2}{3}|< t \, 2^{3j} \, |x|^{-1}\},\\
	\Omega_2 & := \{ \xi\in \R \mid |\xi|<100,\ |\xi-\frac{2}{3}|\geq t \, 2^{3j} \, |x|^{-1}\}.
	\end{align*}
	The integral over $\Omega_1$, which we will call $\mathcal{I}_j^1$, admits direct integration:
	\[ |\mathcal{I}_{j}^1 (t,x,y)|\leq t^{-1/2}\, \int_{\Omega_1}  \frac{1}{ |\frac{1}{4}-\frac{3}{8}\xi|^{1/2}} \, d\xi \lesssim t^{-1/2} \, (t \, 2^{3j} \, |x|^{-1})^{1/2} = |x|^{-1/2}\, 2^{3j/2}.\]
	For  the integral over $\Omega_2$, named $\mathcal{I}_j^2$, we use the fact that $|\phi'(\xi)|\gtrsim |x|$. Indeed, the first term in $\phi'(\xi)$ has size $|x|$, the second has size $|t\, P'(\xi)|\sim t$, and the latter has size $t\, 2^{2j}$. This final fact follows from the fact that in order to have a stationary point,
	\[ |y|\sim t\, 2^{j}\, \Big |\frac{1}{4}-\frac{3}{8}\xi \Big | .\]
After integrating by parts, and using the fact that $t\geq 2^{-j}$, we obtain
\[ |\mathcal{I}_j^2 (t,x,y)|\lesssim |x|^{-1} \, t^{-1/2} \, \sup_{\xi\in\Omega_2}\, \Big |\frac{1}{4}-\frac{3}{8}\xi \Big |^{-1/2}\lesssim |x|^{-1/2}\, t^{-1} \, 2^{-3j/2} \lesssim |x|^{-1/2}\, 2^{-j/2}.\]

Now suppose that there is no stationary point.  In that case, we rewrite the $\mu$-integral (called $f$) as in \eqref{eq:muintegral}, which yields the bounds
\begin{equation}
\label{eq:paf}
|f(t,\xi,y)| \lesssim 2^j,\qquad|\pa_{\xi} f(t,\xi, y)| \lesssim \frac{2^{j}}{|\frac{1}{4}-\frac{3}{8}\xi|} .
\end{equation}
One can also use integration by parts in $f$, as in \eqref{eq:muintegral}, together with the fact that there is no stationary point to obtain the improved estimate:
\begin{equation}\label{eq:paf2}
 |f(t,\xi,y)|\lesssim \frac{2^{-j}}{t |\frac{1}{4}-\frac{3}{8}\xi|}.
\end{equation}

Then the full integral we wish to estimate is
\[ I_j (t,x,y) = \int_{\R} e^{ix\xi + t P(\xi)}\, f(t,\xi,y)\, \chi_0 (\xi)\, d\xi,\]
which we again divide into two regions
	\begin{align*}
	\Omega_1 &  := \{ \xi\in \R \mid |\xi|<100,\ |\xi-\frac{2}{3}|<2^{j/2} \, |x|^{-1/2}\},\\
	\Omega_2 & := \{ \xi\in \R \mid |\xi|<100,\ |\xi-\frac{2}{3}|\geq 2^{j/2} \, |x|^{-1/2}\},
	\end{align*}
with the corresponding integrals being $I_j^1$ and $I_j^2$.

One can use the bound $|f|\lesssim 2^j$ to obtain:
\[ | I_j^1 (t,x,y)| \lesssim \int_{\Omega_1} \, 2^j \, d\xi = 2^{3j/2}\, |x|^{-1/2}.\]

To deal with $\Omega_2$, we differentiate two cases:
\begin{itemize}
\item When $t\, 2^j \lesssim |x|$, we have that the derivative of the phase of $I_j^2$ satisfies:
\[ | x + t P'(\xi)|\gtrsim |x|.\]
Thus integration by parts, together with estimate \eqref{eq:paf} yields
\[  | I_j^2 (t,x,y)| \lesssim |x|^{-1} \, \int_{\Omega_2}  \frac{2^{j}}{|\frac{1}{4}-\frac{3}{8}\xi|} \, d\xi \lesssim |x|^{-1/2} \, 2^{j/2}.\]
\item When $t\, 2^j \gtrsim |x|$, then we may use the improved bound \eqref{eq:paf2} to obtain
\[ | I_j^2 (t,x,y)|  \lesssim \int_{\Omega_2} \frac{2^{-j}}{t |\frac{1}{4}-\frac{3}{8}\xi|}\, d\xi \lesssim \frac{2^{-j}}{t \,  2^{j/2} \, |x|^{-1/2}}
\lesssim 2^{-j/2} \, |x|^{-1/2}.\]
\end{itemize}
	
	The maximal function estimates for $W_0$ in the space $L^4_y L^{\infty}_{T,x}$ can be derived using similar ideas. We summarize these results in the following: 
	
	\begin{thm}\label{thm:maximal5}
	  For any $s>3/4$,
		\begin{align*}
			\norm{W_0 (t) u_0}_{L^4_x L^{\infty}_{t,y}} + \norm{W_0 (t) u_0}_{L^4_y L^{\infty}_{t,x}} & \lesssim \norm{\langle\nabla\rangle^s \, u_0}_{L^2_{x,y}},\\
			\norm{\int_0^t W_0 (t-t') F(t') \, dt'}_{L^4_x L^{\infty}_{T,y}} +\norm{\int_0^t W_0 (t-t') F(t') \, dt'}_{L^4_y L^{\infty}_{T,x}}  & \lesssim \norm{\langle\nabla\rangle^s F}_{L^1_T L^2_{x,y}}.
		\end{align*}
	\end{thm}
	
	\section{Contraction mapping argument}\label{sec:contraction}
	
	In this section we prove Theorem \ref{thm:maintheorem}. Define projections $P_0$, $P_1$ and $P_2$ in frequency space, corresponding to $\chi_0$, $\chi_1$ and $\chi_2$, respectively. We will also need to define the Fourier multiplier operator $P_3:=\pa_x |\nabla|^{-1}$. For $s>1$ and fixed $T>0$, we define the norms:
	\begin{align*}
		\eta_1 (u) & := \norm{ \langle \nabla\rangle^{s+1} P_1 u}_{L^{\infty}_x L^2_{T,y}}+\norm{ \langle \nabla\rangle^{s+1} P_1\, P_3 u}_{L^{\infty}_x L^2_{T,y}} ,\\
		\eta_2 (u) & := \norm{ \langle \nabla\rangle^{s+1} P_2 u}_{L^{\infty}_y L^2_{T,x}}+\norm{ \langle \nabla\rangle^{s+1} P_2 \, P_3 u}_{L^{\infty}_y L^2_{T,x}} ,\\
		\eta_3 (u) & := \norm{\langle \nabla\rangle^{s} u}_{L^{\infty}_T L^{2}_{x,y}},\\
		\eta_4 (u) & :=   \sum_{i=0}^2 \norm{\langle \nabla\rangle^{1/4} P_i u}_{L^{4}_x L^{\infty}_{T,y}},\\
		\eta_5 (u) & :=  \sum_{i=0}^2 \norm{\langle \nabla\rangle^{1/4} P_i u}_{L^{4}_y L^{\infty}_{T,x}},\\
		\eta_6 (u) & := \norm{\langle \nabla\rangle^{s} P_1 u}_{L^q_T L^r_{x,y}}+  \norm{\langle \nabla\rangle^{s} P_2 u}_{L^q_T L^r_{x,y}}
	\end{align*}
	for $(q,r)$ very close to $(2,\infty)$ as given in \Cref{thm:Strichartz2}.
	
	We define $\Lambda_T (u):=\max_{j=1,\ldots,6} \eta_j(u)$ and consider the space
	\[ X_T^s := \{ u \in L^{\infty}_T H^s_{x,y} \mid \Lambda (u)<\infty\}.\]
	We will consider a ball in this space
	\[ B_R= \{ u\in X_T^s \mid \Lambda(u)<R\}\]
	for some $R>0$ to be decided later.
	
	Let us recall the precise nonlinearity in the Dysthe equation:
	\[ N(u)= -\frac{i}{2} |u|^2\, u - \frac{3}{2} |u|^2 \, \pa_x u - \frac{1}{4} u^2 \,\pa_x \overline{u} + \frac{i}{2}\, u\, \pa_x^2 \, |\nabla|^{-1} (|u|^2).\]
	We define the functional
	\begin{equation*}
	 \Psi (u)(t) := W(t) u_0  + \int_0^t W(t-t')N(u(t'))\, dt',
	 \end{equation*}
	 so that a fixed point of $\Psi$ is the solution we seek by the Duhamel formula. Our goal is to show that $\Psi: B_R \rightarrow B_R$ (for some $T$ small enough) and that this mapping is Lipschitz.

	By \Cref{2DDystheStrichartzPropLarge}, 
	\Cref{thm:linearsmoothing}, 
 	\Cref{thm:maximal}, \Cref{thm:maximal2},  
 	\Cref{thm:maximal3}, \Cref{thm:maximal4},
 	and \Cref{thm:maximal5} 
	we have that
	\begin{align}
		\eta_j (\Psi (v)) & \lesssim \norm{\langle \nabla\rangle^s u_0}_{L^2_{x,y}} + \norm{\langle\nabla\rangle^s N(u)}_{L^1_T L^2_{x,y}} \nonumber \\
		& \lesssim  \norm{\langle \nabla\rangle^s u_0}_{L^2_{x,y}} + T^{1/2}\, \norm{\langle\nabla\rangle^s N(u)}_{L^2_{T,x,y}}.\label{eq:contraction}
	\end{align}
	
	Since $N(u)$ is a sum of four terms, let us start by considering the term $|u|^2\, \pa_x u$.  We write
		\begin{multline*}
		\langle\nabla\rangle^s ( |u|^2\, \pa_x u) =  (\langle\nabla\rangle^s \pa_x u )\, |u|^2+ (\langle\nabla\rangle^s |u|^2)\, \pa_x u\\
		+ \left[ \langle\nabla\rangle^s ( |u|^2\, \pa_x u)-  (\langle\nabla\rangle^s |u|^2)\, \pa_x u- (\langle\nabla\rangle^s \pa_x u )\, |u|^2\right] = I + I\! I + I\! I\! I.
	\end{multline*}
	
	We need to control each of these three terms in terms of the $\eta_j$'s.
	\begin{enumerate}
		\item We decompose 
		\[ I=(\langle\nabla\rangle^s \pa_x u )\, |u|^2=\sum_{i=0}^2  (\langle\nabla\rangle^s \pa_x P_i u )\, |u|^2\]
		and treat each summand separately. First, we use the Holder inequality:
		\[ \norm{(\langle\nabla\rangle^s \pa_x P_1 u )\, |u|^2}_{L^2_{T,x,y}}\lesssim \eta_1(u)\, \eta_4 (u)^2.\]
		Similarly, 
		\[ \norm{(\langle\nabla\rangle^s \pa_x P_2 u )\, |u|^2}_{L^2_{T,x,y}}\lesssim \eta_2(u)\, \eta_5 (u)^2.\]
		Finally, for $P_0$ we use \Cref{thm:lowfreq} as follows:
		\begin{align*}
			\norm{(\langle\nabla\rangle^s \pa_x P_0 u )\, |u|^2}_{L^2_{T,x,y}} & \lesssim \norm{\langle\nabla\rangle^s \pa_x P_0 u}_{L^{\infty}_x L^2_{T,y}} \, \eta_4 (u)^2\\
			& \lesssim \norm{\langle\nabla\rangle^s P_0 u}_{L^2_{T,x,y}}\, \eta_4 (u)^2 \lesssim T^{1/2}\, \eta_3(u)\, \eta_4 (u)^2.
		\end{align*}
		\item Now we consider the term $ I\! I= (\langle\nabla\rangle^s |u|^2)\, \pa_x u$. First we use the Holder inequality:
		\[ \norm{(\langle\nabla\rangle^s |u|^2)\, \pa_x u}_{L^2_{T,x,y}}\lesssim \norm{\langle\nabla\rangle^s |u|^2}_{L^{\infty}_T L^2_{x,y}} \, \norm{\pa_x u}_{L^2_T L^{\infty}_{x,y}}.\]
		For the first factor, we use \Cref{thm:chainrule} and the Sobolev embedding theorem:
		\[  \norm{\langle\nabla\rangle^s |u|^2}_{L^{\infty}_T L^2_{x,y}}\lesssim T^{1/2}\, \norm{u}_{L^{\infty}_{T,x,y}} \, \norm{\langle\nabla\rangle^s u}_{L^{\infty}_T L^2_{x,y}}\lesssim T^{1/2} \, \eta_3(u)^2. \]
		For the second factor, we use \Cref{2DDystheStrichartzPropLarge} to control the terms in $P_1$ and $P_2$, and \Cref{thm:lowfreq} for the term in $P_0$:
		\begin{equation}
		\label{eq:example}
			\norm{\pa_x u}_{L^2_T L^{\infty}_{x,y}}  \leq  \norm{\pa_x P_0 u}_{L^2_T L^{\infty}_{x,y}} +  \sum_{i=1}^2 \norm{\pa_x P_i u}_{L^2_T L^{\infty}_{x,y}}  \lesssim \eta_3 (u) + \eta_6 (u).
		\end{equation}
		\item Finally, let us study the error term $I\! I\! I$. By the fractional Leibniz rule, \Cref{thm:Leibnizrule}, we have:
		\[ \norm{I\! I\! I}_{L^2_{x,y}} \lesssim \norm{\langle\nabla\rangle^s |u|^2}_{L^2_{x,y}} \, \norm{\pa_x u}_{L^{\infty}_{x,y}}.\]
		We may then use the Holder inequality for the $L^2_T$ norm to reduce this to case (2).
	\end{enumerate}
	
This shows how to handle the nonlinear term $|u|^2\, \pa_x u$. The term $u^2 \, \pa_x \bar{u}$ is analogous, and the term $|u|^2\, u$ is trivial, so we skip them. 

We now focus on the term $u\, \pa_x^2 \, |\nabla|^{-1} ( |u|^2)$. Recall that we defined $P_3=\pa_x |\nabla|^{-1}$, which is a pseudo-differential operator of order zero, and thus maps $L^p_{x,y}$ to $L^p_{x,y}$ continuously for any $1<p<\infty$ (by the H\"ormander-Mikhlin multiplier theorem). We can also write 
\[ u\, P_3\, \pa_x ( |u|^2) = u\, P_3\, \pa_x( u \, \bar{u})= u\, P_3\left( \pa_x u \, \bar{u} + u\, \pa_x \bar{u}\right).\]
The last two terms admit a similar treatment and therefore we will discuss only the case of $u\, P_3 ( \pa_x u \, \bar{u} )$.

We go back to \eqref{eq:contraction} and plug in this term. We need to estimate the $L^2_{T,x,y}$ norm of $\langle \nabla\rangle^s [u\, P_3 ( \pa_x u \, \bar{u} )]$. Once again, we decompose:
	\begin{multline*}
		\langle \nabla\rangle^s [u\, P_3 ( \pa_x u \, \bar{u} )] =  (\langle\nabla\rangle^s u )\, P_3 ( \pa_x u \, \bar{u} )+ u\, \langle\nabla\rangle^s\,P_3 ( \pa_x u \, \bar{u} )\\
	 + \left[\langle \nabla\rangle^s [u\, P_3 ( \pa_x u \, \bar{u} )] -  (\langle\nabla\rangle^s u )\, P_3 ( \pa_x u \, \bar{u} )- u\, \langle\nabla\rangle^s\,P_3 ( \pa_x u \, \bar{u} )\right] = I + I\! I + I\! I\! I.
	\end{multline*}
We estimate each of these terms separately:
	\begin{enumerate}
	\item By the Holder inequality, the Sobolev embedding theorem and the fact that $P_3$ maps $L^r_{x,y}$ to $L^r_{x,y}$ continuously,
	\begin{align*}
	 \norm{I}_{L^2_{T,x,y}} & \lesssim \eta_3 (u)\, \norm{  P_3 ( \pa_x u \, \bar{u} )}_{L^2_T L^{\infty}_{x,y}} \lesssim \eta_3 (u) \, \norm{ \langle\nabla\rangle^{\varepsilon} P_3 ( \pa_x u \, \bar{u} )}_{L^2_T L^{r}_{x,y}}\\
	& \lesssim \eta_3 (u) \, \norm{ \langle\nabla\rangle^{\varepsilon} ( \pa_x u \, \bar{u} )}_{L^2_T L^{r}_{x,y}}.
	\end{align*}
	Now we write:
	\begin{multline*}
	  \langle\nabla\rangle^{\varepsilon} ( \pa_x u \, \bar{u} )  =  (\langle\nabla\rangle^{\varepsilon}  \pa_x u) \, \bar{u} + \pa_x u\, (\langle\nabla\rangle^{\varepsilon}  \bar{u}) \\
	 + [  \langle\nabla\rangle^{\varepsilon} ( \pa_x u \, \bar{u} )  - (\langle\nabla\rangle^{\varepsilon}  \pa_x u) \, \bar{u} - \pa_x u\, (\langle\nabla\rangle^{\varepsilon}  \bar{u})] = I_1 + I_2 + I_3.
	\end{multline*}
	\begin{enumerate}
	\item To control the first term we use the Holder inequality and the Sobolev embedding theorem:
	\[ \norm{I_1}_{L^2_T L^r_{x,y}} \lesssim \norm{\bar{u}}_{L^{\infty}_{T,x,y}} \, \norm{ \langle\nabla\rangle^{\varepsilon}  \pa_x u}_{L^2_T L^r_{x,y}}\lesssim \eta_3 (u)\, \norm{ \langle\nabla\rangle^{\varepsilon}  \pa_x u}_{L^2_T L^r_{x,y}}.\]
	The last factor is bounded by $T^{0+}\,\left( \eta_3(u) + \eta_6 (u)\right)$ after breaking it up using $P_i$ ($i=0,1,2$) and \Cref{thm:lowfreq}.
	\item The second term is analogous. Indeed,
	\[ \norm{I_2}_{L^2_T L^r_{x,y}} \lesssim \norm{\pa_x u}_{L^q_2 L^r_{x,y}} \, \norm{\langle\nabla\rangle^{\varepsilon} \bar{u}}_{L^{\infty}_{T,x,y}},\]
	and one proceeds as with $I_1$.
	\item The error term, $I_3$, admits the same control as $I_1$ after using the fractional Leibniz rule (\Cref{thm:Leibnizrule}).
	\end{enumerate}
	
	\item Regarding $I\! I$, we decompose:
	\begin{multline*}
	I\! I  = |u|^2 \, (\langle\nabla\rangle^s \, P_3 \pa_x u) + u \,(\langle\nabla\rangle^s\bar{u})  \, ( P_3 \, \pa_x u) \\
	 + \left[ u\, \langle\nabla\rangle^s\,P_3 ( \pa_x u \, \bar{u} ) - |u|^2 \, (\langle\nabla\rangle^s P_3 \, \pa_x u)-  u \,(\langle\nabla\rangle^s\bar{u})  \, ( P_3\, \pa_x u) \right] = I\! I_1 + I\! I_2 + I\! I_3.
	\end{multline*}
	We treat each case separately:
	\begin{enumerate}
	\item We write 
	\[ \langle\nabla\rangle^s \,P_3 \pa_x u = \sum_{i=0}^2 P_i\, \langle\nabla\rangle^s \, P_3 \, \pa_x u.\]
	The terms in $P_1$ and $P_2$ are analogous so we do only one. By the Holder inequality,
	\[ \norm{|u|^2 \, (P_1\, \langle\nabla\rangle^s \, P_3 \pa_x u)}_{L^2_{T,x,y}}\lesssim \eta_4 (u)^2 \, \eta_1(u).\]
	The term in $P_0$ is controlled with $\eta_3(u)$ thanks to \Cref{thm:lowfreq} (whose proof is identical when including $P_3$).
	\item We control this term using the Holder inequality:
	\begin{align*}
	 \norm{I\! I_2}_{L^2_{T,x,y}} & \lesssim \norm{u}_{L^{\infty}_{T,x,y}}\, \norm{\langle\nabla\rangle^s u}_{L^{\infty}_T L^2_{x,y}}\, \norm{P_3 \, \pa_x u}_{L^2_T L^{\infty}_{x,y}}\\
	& \lesssim \eta_3 (u)^2\, \norm{P_3 \, \pa_x u}_{L^2_T L^{\infty}_{x,y}}.
	\end{align*}
	The last factor can be controlled as in \eqref{eq:example}.
	\item We first use the H\"older inequality to write:
	\[ \norm{I\! I_3}_{L^2_{T,x,y}} \lesssim \norm{u}_{L^{\infty}_{T,x,y}}\, \norm{\mbox{error}}_{L^2_{T,x,y}}.\]
	The error term is absolutely analogous to case $I\! I_2$ after using the fractional Leibniz rule, \Cref{thm:Leibnizrule}.
	\end{enumerate}
	\item This final error term, $I\! I \! I$, is handled as $I$ thanks to \Cref{thm:Leibnizrule} again.
	\end{enumerate}
	
	Back to \eqref{eq:contraction}, these arguments show that 
	\[ \Lambda ( \Psi (u)) \lesssim \norm{\langle\nabla\rangle^s u_0}_{L^2_{x,y}} + T^{1/2}\, \langle T\rangle^{0+}\, \Lambda (u)^3.\]
	Recall that we are working with $u$ in the ball $B_R$ of functions with $\Lambda (u)\leq R$. One can choose $R$ large enough so that $\norm{\langle\nabla\rangle^s u_0}_{L^2_{x,y}}\leq \frac{R}{2}$ and take $T$ small enough to guarantee that $C\,T^{1/2}\, \langle T\rangle^{0+} \, R^3\leq \frac{R}{2}$. With this choice, $\Psi$ maps $B_R$ to $B_R$.
	
	One can use similar ideas to prove that $\Psi$ is contraction, i.e. 
	\[ \Lambda \left( \Psi(u)-\Psi (v) \right) \lesssim T^{1/2}\, \norm{\langle\nabla\rangle^s [ N(u) - N(v)]}_{L^2_{T,x,y}}\lesssim T^{1/2}\, \langle T\rangle^{0+} \, R^2\, \Lambda (u-v).\]
	One can choose a smaller $T$, if necessary, to guarantee that $C\,T^{1/2}\, \langle T\rangle^{0+} \, R^2<1$. This finishes the proof.

\section{Ill-posedness}\label{sec:illposedness} 

\subsection{Main idea} 

In this section we prove \Cref{thm:illposedness}. This is a mild form of ill-posedness, which first appeared in the work of Bourgain \cite{Bourgain}. The intuition behind this idea is the following: consider the Dysthe equation \eqref{eq:Dysthe} for perturbed initial data $\epsilon\, u_0$. We may then write an asymptotic expansion for the solution to this problem, $u_{\epsilon}$, in powers of $\epsilon$:
\[ u_{\epsilon}= \epsilon u_1+\epsilon^2 u_2 + \ldots \]
By plugging this into \eqref{eq:Dysthe} and matching the coefficients of the powers of $\epsilon$, one can write explicit equations for $u_1$, $u_2$, etc. In particular, one finds that $u_1=W(t)u_0$, $u_2=0$ and $u_3$ solves the equation:
\begin{equation*}
\left\{
    \begin{array}{ll}
    \pa_t u_3 + L(u_3) = N(u_1),\\
    u_3|_{t=0} =0,
    \end{array}
\right.
\end{equation*}
for $L$ and $N$ as in \eqref{eq:Dysthe}. This formal procedure is equivalent to considering the first nontrivial term of a Picard iterative scheme for the Dysthe equation \eqref{eq:Dysthe}, which would be
\begin{equation}\label{eq:Picarditeration}
 W(t) u_0 + \int_0^t W(t-t') N( W(t')u_0)\, dt' = u_1 (t) + u_3 (t).
\end{equation}

Our goal in this section is to show that the operator that maps initial data $u_0$ to $u_3$ is not continuous from $H^s (\R^2)$ to $C([0,T],H^s (\R^2))$ for $s<0$, no matter  how small $T$ is. Note that the existence of $u_3$ and some small enough $T$ is guaranteed by estimates similar to those used in the proof of \Cref{thm:maintheorem}. The lack of continuity of the map from $u_0$ to $u_3$ is equivalent to the fact that the map from $u_0\in H^s(\R^2)$ to the solution to the Dysthe equation \eqref{eq:Dysthe} in $L^{\infty}([0,T],H^s(\R^2))$ is not $C^3$. Moreover, this means that any attempt to prove local well-posedness for the Dysthe equation based on an iterative scheme in $H^s (\R^2)$ like the one described in \eqref{eq:Picarditeration} must necessarily fail.

Let us briefly discuss some motivation to justify why we obtain ill-posedness for $s<0$. As explained before, one issue with the Dysthe equation is the lack of scaling symmetry, so we cannot technically talk about a critical regularity $s_c$ that is invariant under rescaling. When such symmetry is available, the connection between the criticality of the problem and scaling is the following: 
\begin{itemize}
\item in the subcritical case $s>s_c$, we expect high frequencies to evolve linearly for all times, while the low-frequencies will evolve linearly for small times and nonlinearly for large times.
\item in the supercritical case $s<s_c$, high frequencies are unstable and develop nonlinear behavior in short times.
\end{itemize}
See Principle 3.1 in Tao's book \cite{tao} and the discussion that follows for more details. Despite the lack of scaling symmetry, the same heuristics can be applied to our equation: we expect the largest contribution to high-frequencies to come from the terms in \eqref{eq:Dysthe} involving the largest number of derivatives. Therefore, a reasonable model to understand the behavior of large frequencies (at least for short times) might be the following PDE:
\[\left\{
    \begin{array}{ll}
    \pa_t u -\frac{1}{16} \pa_x^3 u + \frac{3}{8}\pa_x\pa_y^2 u =  - \frac{3}{2} |u|^2 \, \pa_x u - \frac{1}{4} u^2 \,\pa_x \overline{u} + \frac{i}{2}\, u\, \pa_x^2 |\nabla|^{-1} (|u|^2),\\
    u|_{t=0} =u_0,
    \end{array}
\right.\]
This PDE does enjoy a scaling symmetry and its critical regularity under it is precisely $s_c=0$. In the next section, we will see that the terms in this equation constitute the top order of our approximation to $u_3$ for short times, which might explain the range of $s$ in \Cref{thm:illposedness}

\subsection{Computations}

Consider initial data $u_{0,N}$ with small support around some high frequency $N=(N_1,N_2)\in \R^2$. In particular, consider 
\[  \widehat{u}_{0,N}(\xi,\mu)= c(N)\,  f(\xi) \, g(\mu)\]
where
\begin{itemize}
\item $f$ is an odd function such that $f=0$ outside the interval $[N_1 - N_1^{2 \varepsilon} , N_1 +N_1^{2 \varepsilon}]$ for some $\varepsilon<1/2$, and $f=1$ in the interval $[N_1 - \frac{1}{2}\, N_1^{2 \varepsilon} , N_1 +\frac{1}{2}\, N_1^{2 \varepsilon}]$.
\item $g$ is an even function such that $0\leq g\leq 1$, and with similar properties in the interval $[N_2 - N_2^{2 \varepsilon} , N_2 + N_2^{2 \varepsilon}]$.
\item $c(N)$ is a coefficient chosen to normalize the $H^s_{x,y}$-norm of $u_{0,N}$, i.e.
\[ c(N):= C\, N_1^{-\varepsilon} \,  N_2^{-\varepsilon} \, (\max\{N_1^s, N_2^s\})^{-1}.\]
\end{itemize}

We first consider the linear flow $ u_{1,N}= W(t) u_{0,N}$. Our goal is to approximate the function
\begin{equation}
\label{eq:defu3}
 u_{3,N}(t,x,y) = \int_0^t W(t-t') \, N(u_{1,N}(t') )\, dt'.
 \end{equation}

The main argument will be as follows: in order for the map from $u_0$ to $u_3$ to be continuous we need
\[ \norm{u_{0,N}}_{H^s_{x,y}} \gtrsim \norm{u_{3,N}}_{L^{\infty}([0,t], H^s_{x,y})}\gtrsim \norm{\langle\xi\rangle^s \widehat{u}_{3,N}}_{L^{\infty}([0,t], L^2(\mathcal{R}_N))}.\]
Here, $\mathcal{R}_N$ is a region where we will be able to approximate $\widehat{u}_{3,N}$ accurately.

First of all, we compute the linear flow for our choice of initial data:
\begin{equation}\label{eq:firstiteration}
 \widehat{u}_{1,N}(t,\xi,\mu) = e^{-it\, \Phi (\xi,\mu)}\, c(N)\, f(\xi) \, g(\mu).
\end{equation}
Now we wish to compute $N(u_{1,N})$. As an example, let us start by computing the contribution of the term $|u_{1,N}|^2\, \pa_x u_{1,N}$. Recall that
\[ (|u|^2 \pa_x u)^{\wedge}= \widehat{u} \ast \widehat{\bar{u}}\ast (-i\,\xi\, \widehat{u}).\]
Using this, together with \eqref{eq:firstiteration}, we may write:
%
%
\begin{align}\label{eq:firstiteration2}
(|u_{1,N}|^2 \pa_x u_{1,N})^{\wedge}(t,\xi,\mu) & = i\,c(N)^{3}\, \int_{\R^4}  e^{-it \, \Omega}\, (\xi-\xi_1+\xi_2) \,f(\xi-\xi_1+\xi_2) \, g(\mu-\mu_1+\mu_2)  \\
& \hspace{3cm} f(\xi_1) \, g(\mu_1) \,  f(\xi_2) \, g(\mu_2)\, d\xi_2 \, d\mu_2\, d\xi_1\,d\mu_1,\nonumber
\end{align}
where 
\begin{equation*}
\Omega = \Phi (\xi-\xi_1+\xi_2,\mu-\mu_1+\mu_2) +\Phi (\xi_1,\mu_1) - \Phi (\xi_2,\mu_2).
\end{equation*}

We take its Fourier transform of \eqref{eq:defu3} and compute the $t'$-integral explicitly using \eqref{eq:firstiteration2}:
\begin{align}\label{eq:firstiteration3}
\widehat{u}_{3,N}(t,\xi,\mu) & = i\,e^{-it\Phi(\xi,\mu)}\, c(N)^{3}\, \int_{\R^4}  \frac{e^{-it \, \widetilde{\Omega}}-1}{-i\widetilde{\Omega}} \, (\xi-\xi_1+\xi_2) \,f(\xi-\xi_1+\xi_2) \, g(\mu-\mu_1+\mu_2) \\
&  \hspace{2cm}  f(\xi_1) \, g(\mu_1) \,  f(\xi_2) \, g(\mu_2)\, d\xi_2 \, d\mu_2\, d\xi_1\,d\mu_1 + \mbox{other terms},\nonumber
\end{align}
where
\begin{equation}\label{eq:Omega}
\widetilde{\Omega} = \Phi (\xi-\xi_1+\xi_2,\mu-\mu_1+\mu_2)-\Phi(\xi,\mu) +\Phi (\xi_1,\mu_1) - \Phi (\xi_2,\mu_2).
\end{equation}

We now develop a rigorous approximation to \eqref{eq:firstiteration3} for short times.

\begin{lem}\label{thm:approx1} Suppose that $|t\widetilde{\Omega}|<1/4$ and that $(\xi,\mu)\in B(N_1, \frac{1}{4} N_1^{2\varepsilon}) \times B(N_2, \frac{1}{4} N_2^{2\varepsilon})$. Then for large enough $N$, we have that
\begin{equation*}
 \widehat{u}_{3,N}(t,\xi,\mu)  = -i\,\frac{7}{4}\, c(N)^{3}\, t\,  e^{-it\Phi(\xi,\mu)}\, F(\xi)\, G(\mu)  + \mbox{error},
 \end{equation*}
where 
\begin{align*}
F(\xi) & = \int_{\R^2} (\xi-\xi_1+\xi_2) \,f(\xi-\xi_1+\xi_2)\, f(\xi_1) \,  f(\xi_2)\, d\xi_1 \, d\xi_2,\\
G(\mu) & = \int_{\R^2}  g(\mu-\mu_1+\mu_2) \, g(\mu_1) \, g(\mu_2)\, d\mu_1 \, d\mu_2.\\
\end{align*}
\end{lem}
\begin{proof}
{\bf Step 1.} First we use a Taylor expansion:
\[ \frac{e^{-it \, \widetilde{\Omega}}-1}{-i\widetilde{\Omega}}  = \frac{1}{-i\widetilde{\Omega}}\, \sum_{k=1}^{\infty} \frac{(-it \, \widetilde{\Omega})^{k}}{k!}= t + \sum_{k=2}^{\infty} \frac{(-t)^k \, (-i\,\widetilde{\Omega})^{k-1}}{k!}= t + \O ( t^2\, |\widetilde{\Omega}|).\]
since
\[ \Big | \sum_{k=2}^{\infty} \frac{(-t)^k \, (-i\,\widetilde{\Omega})^{k-1}}{k!} \Big | \leq t\, \sum_{k=1}^{\infty} |t\widetilde{\Omega}|^{k} \leq t\, \frac{|t\widetilde{\Omega}|}{1- |t\widetilde{\Omega}|}\leq \frac{4}{3}\, t^2\, |\widetilde{\Omega}|<\frac{1}{3} t.\]

Consequently, the contribution from the term $|u|^2\, \pa_x u$ to \eqref{eq:firstiteration3} can be rewritten as
\[i\,c(N)^{3}\, t\,  e^{-it\Phi(\xi,\mu)}\, F(\xi)\, G(\mu) + R(t,\xi,\mu),\]
where
\begin{align*}
 |R(t,\xi,\mu)| & \leq  \frac{4}{3}\, c(N)^{3}\, t^2 \, \int_{\R^4} |\widetilde{\Omega}| \, |\xi-\xi_1+\xi_2| \,f(\xi-\xi_1+\xi_2)\, f(\xi_1) \, f(\xi_2)\\
& \hspace{3cm} g(\mu-\mu_1+\mu_2) \, g(\mu_1) \, g(\mu_2)\, d\mu_1 \, d\mu_2\, d\xi_1 \, d\xi_2.
 \end{align*}
Using the fact that $|t\widetilde{\Omega}|<\frac{1}{4}$ we have that
 \[ |R(t,\xi,\mu)|\leq \frac{1}{3} \, t \, c(N)^3\, \tilde{F}(\xi)\, G(\mu),\]
 where
 \[ \widetilde{F}(\xi) := \int_{\R^2} |\xi-\xi_1+\xi_2| \,f(\xi-\xi_1+\xi_2)\, f(\xi_1) \,  f(\xi_2)\, d\xi_1 \, d\xi_2.\]
 It is now easy to guarantee that this error does not change the top order behavior, see \Cref{rk:positivity} for details.
 
{\bf Step 2.} Similar techniques show that the term nonlinear term $|u|^2 u$ produces a negligible contribution. Let us now consider the nonlinear term $u^2 \, \pa_x \bar{u}$. 
A similar procedure as the one developed in \eqref{eq:firstiteration3} for $|u|^2\, \pa_x u$ shows that the contribution of $u^2 \, \overline{\pa_x u}$ to $\widehat{u_{3,N}}$ is:
\[ -i\,c(N)^3 \, t \, e^{-it\, \Phi(\xi,\mu)}\, H(\xi)\, S(\mu) + \mbox{error}\]
for 
\begin{align*}
H(\xi) & := \int_{\R^2} (-\xi+\xi_2+\xi_1)\, f(-\xi+\xi_2+\xi_1)\, f(\xi_1) \,  f(\xi_2) \, d\xi_1 \, d\xi_2,\\
S(\mu) & := \int_{\R^2} g(-\mu+\mu_2+\mu_1)\, g(\mu_1)\, g(\mu_2)\, d\mu_1\, d\mu_2.
\end{align*}
The error term admits a similar analysis to that of Step 1.
The key point now is that $H(\xi)=-F(\xi)$ and $S(\mu)=G(\mu)$ thanks to the fact that $f$ is odd and $g$ is even. This guarantees that this term and the top order term from Step 1 do not cancel. In fact the two terms have the same sign and we will have a top order of:
\[ -i\,\frac{7}{4} \, c(N)^3 \, t \, e^{-it\, \Phi(\xi,\mu)}\, F(\xi)\, G(\mu).\]

{\bf Step 3.} The same ideas can be used to approximate the contribution of $u\, \pa_x^2 |\nabla|^{-1} (|u|^2)$ leading to a top order term given by
\[ \frac{i}{2} \, c(N)^{3}\, t\,  e^{-it\Phi(\xi,\mu)}\, F^{\ast}(\xi,\mu),\]
 where 
\begin{multline*}
F^{\ast}(\xi,\mu) = \int_{\R^4} \frac{(\xi-\xi_1)^2}{[(\xi-\xi_1)^2+(\mu-\mu_1)^2]^{1/2}} \,f(\xi-\xi_1+\xi_2)\, f(\xi_1) \,  f(\xi_2)\\ g(\mu-\mu_1+\mu_2) \, g(\mu_1) \, g(\mu_2)\, d\mu_1 \, d\mu_2\, d\xi_1 \, d\xi_2 .\\
\end{multline*}
Note however that this term cannot cancel the leading term given that
\[ |F^{\ast}(\xi,\mu)|\leq G(\mu)\, \int_{\R^2} |\xi-\xi_1|\, |f(\xi-\xi_1+\xi_2)|\, |f(\xi_1)| \, |f(\xi_2)|\, d\xi_1, d\xi_2 ,\]
and that the last integral has size $N_1^{6\varepsilon}$. This is therefore a lower order term (see \Cref{thm:approx3} below).
\end{proof}

\

Now we study the size of $\widetilde{\Omega}$ as defined in  \eqref{eq:Omega} (the same applies to that of $\widetilde{\omega}$). An application of the mean value theorem yields the following

\begin{lem}\label{thm:approx2} Suppose that $(\xi_k,\mu_k)\in B(N_1, N_1^{2\varepsilon}) \times B(N_2, N_2^{2\varepsilon})$ for $k=1,2$. Then for $(\xi,\mu)\in B(N_1, N_1^{2\varepsilon}) \times B(N_2, N_2^{2\varepsilon})$ we have that 
\[ |\widetilde{\Omega}|\lesssim (N_1^2 + N_2^2) \, N_1^{2\varepsilon} + N_1 \, N_2^{1+ 2\varepsilon}.\]
\end{lem}

\Cref{thm:approx1} tells us that we will have a valid approximation whenever $|t\widetilde{\Omega}|<1/4$, or more precisely when
\begin{equation}\label{eq:timecond}
 |t|\lesssim \left(\max\{ N_1^{2+2\varepsilon},\ N_2^2 \, N_1^{2\varepsilon},\ N_1 \, N_2^{1+2\varepsilon}\} \right)^{-1}.
\end{equation}

\begin{rk}\label{rk:positivity} We also need to make sure that the top order in \Cref{thm:approx1} controls the error. That can be achieved by restricting $\xi$ to a region where $F(\xi)=\widetilde{F}(\xi)$, so that the error term has at most $\frac{1}{3}$ of the size of the top order. Since $f$ is positive, we only need to make sure that the sign of $\xi-\xi_1+\xi_2$ is positive. However, note that $-\xi_1+\xi_2 \in B(0,2\, N_1^{2\varepsilon})$, and in \Cref{thm:approx2} we already require $\xi \in B(N_1, N_1^{2\varepsilon})$, so $\xi-\xi_1+\xi_2$ must be positive for $N_1$ large enough (and we will take it to infinity).
\end{rk}

We can finally estimate the size of $F$ and $G$. We omit the proof since it simply consists on estimating the contribution of the integrands on their respective supports.

\begin{lem}\label{thm:approx3} Suppose that $(\xi,\mu)\in B(N_1, \frac{1}{4} N_1^{2\varepsilon}) \times B(N_2, \frac{1}{4} N_2^{2\varepsilon})$, then 
\begin{equation*}
F(\xi) \geq \frac{1}{2\cdot 8^2}\, N_1^{1+4\varepsilon},\qquad G(\mu)  \geq \frac{1}{2\cdot 8^2}\, N_2^{4\varepsilon}.
\end{equation*}
\end{lem}

Finally, we can put these results together:
\begin{prop}\label{thm:approx4} Suppose that $(\xi,\mu)\in B(N_1, \frac{1}{4} N_1^{2\varepsilon}) \times B(N_2, \frac{1}{4} N_2^{2\varepsilon})$ and that $t$ satisfies condition \eqref{eq:timecond}. Then
\[ |\widehat{u}_{3,N}(t,\xi,\mu)| \gtrsim  t\, c(N)^{3}\, N_1^{1+4\varepsilon}\, N_2^{4\varepsilon},\]
where the implicit constant is independent of $t,N,\varepsilon$.
\end{prop}
\begin{proof}
As long as $t$ satisfies \eqref{eq:timecond}, we use \Cref{thm:approx1}, and note that the error term is at most $\frac{1}{3}$ the size of the leading order. Then we combine this with the results of \Cref{thm:approx3}.
\end{proof}

We are ready to give the main argument: consider $\mathcal{R}_N= B(N_1, \frac{1}{4} N_1^{2\varepsilon}) \times B(N_2, \frac{1}{4} N_2^{2\varepsilon})$ for $N=(N_1,N_2)$. Suppose that $N_2\ll N_1$ so that the normalization constant is 
\[ c(N)\sim N_1^{-s-\varepsilon} \,  N_2^{-\varepsilon}.\]
Then in order for the initial-data-to-solution map to be $C^3$ in $[0,T]$, we need to have:
\[ 1=\norm{u_{0,N}}_{H^s_{x,y}} \gtrsim \norm{u_{3,N}}_{L^{\infty}([0,T], H^s_{x,y})}\gtrsim \norm{\langle\xi\rangle^s \widehat{u}_{3,N}}_{L^{\infty}([0,T], L^2(\mathcal{R}_N))}. \]
Note that there should exist a common time of existence $T$ for all $N$ as the initial data has size 1.

We can further substitute $T$ by $t$, and take this as close as we want to the upper bound given by \eqref{eq:timecond} (note that such $t$ are smaller than any fixed $T$ for $N$ large enough, and thus valid). Then by \Cref{thm:approx4}, we have that 
\begin{equation*}
 1  \gtrsim \norm{\langle\xi\rangle^s \widehat{u}_{3,N}}_{L^{\infty}([0,T], L^2(\mathcal{R}_N))} \gtrsim t\, c(N)^{3}\, N_1^{1+4\varepsilon}\, N_2^{4\varepsilon} \, N_1^s\, |\mathcal{R}_N|^{1/2} \gtrsim N_1^{-1-2s}\, N_2^{2\varepsilon}.
 \end{equation*}
We set $N_2 = 10^{-5} N_1$ and take $N_1\rightarrow \infty$. A contradiction would be reached unless
\[ -1-2s+2\varepsilon<0\quad \Leftrightarrow\quad s>-\frac{1}{2} + \varepsilon.\]
By taking $\varepsilon:= \frac{1}{2}-$ we obtain the condition $s\geq 0$. 
	
	\appendix

	\section{Technical results}
	
	We use this appendix to write some technical results used throughout the paper. We start with the Van der Corput lemma, as seen in \cite{KPV2}.
	
	\begin{lem}[Van der Corput]\label{thm:VanderCorput} Suppose that $\phi$ is a real-valued $C^2$ function defined in $[a,b]$ such that $|\phi''(x)|>c$ in $[a,b]$. Then
		\[ \Big | \int_a^b e^{it\phi(x)}\, \psi (x)\, dx \Big | \leq 10\, |t\,c|^{-1/2}\, \left( |\psi (b)| + \int_a^b |\psi'(x)|\, dx\right).\]
	\end{lem}
		Note that the right-hand side is controlled by $|tc|^{-1/2} \left( \norm{\psi}_{L^{\infty}} + \norm{\psi'}_{L^1}\right)$. Moreover, if $\psi$ has a finite number of changes of monotinicity, the right-hand side is entirely controlled by $|tc|^{-1/2} \norm{\psi}_{L^{\infty}}$.
	
	We also have the following version of the inverse function theorem. In particular, we will make use of \eqref{eq:invfuncond} to estimate the size of the neighborhood, which can be proved with a simple fixed point argument.
	
		\begin{thm}[Inverse function theorem]\label{thm:inversefunction}
		Suppose that $f: U\rightarrow \R$ is a $C^1$ function such that $f'(x_0)\neq 0$. Then there exists some $\varepsilon>0$ such that
		\[ f:  B(x_0,\varepsilon) \rightarrow B \left(f(x_0),\frac{f'(x_0)}{2}\,\varepsilon \right) \]
		is a $C^1$ diffeomorphism. Moreover, a (non-optimal) $\varepsilon$ can be found by imposing
		\begin{equation}
		\label{eq:invfuncond}
			\Big |\frac{f'(x)}{f'(x_0)} \Big | \geq \frac{1}{2}\quad  \mbox{for all}\ x\in B(x_0,\varepsilon).
		\end{equation}
	\end{thm}

	We also record some results that are useful to work with fractional derivatives. The first one is Theorem A.12 in \cite{KPV}. See also \cite{KatoPonce} for similar estimates, as well as Proposition 3.3 in \cite{CW}.
	
	\begin{thm}[Fractional Leibniz rule]\label{thm:Leibnizrule} 
		Let $s\in (0,1)$ and $1<p<\infty$. Then
		\[ \norm{D_x^s (u\, v) - D_x^s u \, v - u\, D_x^s v}_{L^p_x}\lesssim \norm{u}_{L^{\infty}_x}\, \norm{D_x^{s} v}_{L^p_x}.\]
	\end{thm}
	
	The following result is Proposition 3.1 in \cite{CW}. 
	
	\begin{thm}[Fractional chain rule]\label{thm:chainrule}
		Suppose that $F\in C^1 (\C)$ and $s\in (0,1)$. Let $p,p_1,p_2 \in (1,\infty)$ such that 
		\[ \frac{1}{p}=\frac{1}{p_1} + \frac{1}{p_2}.\] 
		If $u\in L^{\infty}(\R)$, then
		\[ \norm{D_x^s F(u)}_{L^p_x} \lesssim \norm{F'(u)}_{L^{p_1}_x} \, \norm{D_x^s u}_{L^{p_2}_x}.\]
	\end{thm}

	\bibliographystyle{amsplain}
	\bibliography{references}

\providecommand{\bysame}{\leavevmode\hbox to3em{\hrulefill}\thinspace}
\providecommand{\MR}{\relax\ifhmode\unskip\space\fi MR }
\providecommand{\MRhref}[2]{%
  \href{http://www.ams.org/mathscinet-getitem?mr=#1}{#2}
}
\providecommand{\href}[2]{#2}
\begin{thebibliography}{10}

\bibitem{Strich3}
A.~Arnold, J.~Kim, and X.~Yao, \emph{Estimates for a class of oscillatory
  integrals and decay rates for wave-type equations}, Journal of mathematical
  analysis and applications \textbf{394} (2012), no.~1, 139--151.

\bibitem{Bejenaru}
I.~Bejenaru, \emph{Global results for {Schr{\"o}dinger} maps in dimensions
  $n\geq 3$}, Comm. Partial Differential Equations \textbf{33} (2008),
  451--477.

\bibitem{BIK}
I.~Bejenaru, A.~D. Ionescu, and C.~E. Kenig, \emph{Global existence and
  uniqueness of {Schr{\"o}dinger} maps in dimensions $d\geq 4$}, Adv. Math.
  \textbf{215} (2007), 263--291.

\bibitem{BIKT}
I.~Bejenaru, A.~D. Ionescu, C.~E. Kenig, and D.~Tataru, \emph{Global
  {Schr{\"o}dinger} maps in dimensions $d\geq 2$: Small data in the critical
  {Sobolev} spaces}, Annals of Mathematics \textbf{173} (2011), 1443--1506.

\bibitem{BenArtziKochSaut}
M.~Ben-Artzi, H.~Koch, and J.C. Saut, \emph{Dispersion estimates for third
  order equations in two dimensions}, Communications in Partial Differential
  Equations \textbf{28} (2003), no.~11-12, 1943--1974.

\bibitem{Bourgain}
J.~Bourgain, \emph{Fourier transform restriction phenomena for certain lattice
  subsets and applications to nonlinear evolution equations. {Part} {II}: The
  {KdV} equation}, Geometric and Functional Analysis \textbf{3} (1993),
  209--262.

\bibitem{CW}
F.~M. Christ and M.~I. Weinstein, \emph{Dispersion of small amplitude solutions
  of the generalized {Korteweg}-de {Vries} equation}, Journal of Functional
  Analysis \textbf{100} (1991), 87--109.

\bibitem{CouSaps}
W.~Cousins and T.~P. Sapsis, \emph{Reduced-order precursors of rare events in
  unidirectional nonlinear water waves}, J. Fluid Mech. \textbf{790} (2016),
  368--388.

\bibitem{Sulem}
W.~Craig, P.~Guyenne, and C.~Sulem, \emph{Normal form transformations and
  {D}ysthe equation for the nonlinear modulation of deep-water gravity waves},
  Water Waves (2020).

\bibitem{VandenEijnden}
G.~Dematteis, T.~Grafke, and E.~Vanden-Eijnden, \emph{Rogue waves and large
  deviations in deep sea}, PNAS \textbf{115} (2018), no.~5, 855--860.

\bibitem{Strich2}
Yong Ding and Xiaohua Yao, \emph{Lp-lq estimates for dispersive equations and
  related applications}, Journal of Mathematical Analysis and Applications
  \textbf{356} (2009), no.~2, 711--728.

\bibitem{Dysthe}
K.~B. Dysthe, \emph{Note on a modification to the nonlinear {Schr\"{o}dinger}
  equation for application to deep water waves}, Proc. R. Soc. Lond. A
  \textbf{369} (1979), 105--114.

\bibitem{roguereview}
K.~B. Dysthe, H.~E. Krogstad, and P.~M\"{u}ller, \emph{Oceanic rogue waves},
  Annu. Rev. Fluid Mech. \textbf{40} (2008), 287--310.

\bibitem{FarSap}
M.~Farazmand and T.~Sapsis, \emph{Reduced-order prediction of rogue waves in
  two-dimensional deep-water waves}, Journal of Computational Physics
  \textbf{340} (2017), 418--434.

\bibitem{SapsisMechanism}
\bysame, \emph{Extreme events: Mechanisms and prediction}, Applied Mechanics
  Reviews \textbf{71} (2019), no.~5.

\bibitem{Hasselmann}
K.~Hasselmann, \emph{On the non-linear energy transfer in a gravity-wave
  spectrum. {Part} 1: General theory}, J. Fluid Mech. \textbf{12} (1961),
  481--500.

\bibitem{Hormander}
L.~H\"{o}rmander, \emph{The analysis of linear partial differential operators
  ii: Differential operators with constant coefficients}, Classics in
  Mathematics, Springer-Verlag Berlin Heidelberg, 2005.

\bibitem{IonescuKenig}
A.~D. Ionescu and C.~E. Kenig, \emph{Low-regularity {Schr{\"o}dinger} maps,
  {II}: global well-posedness in dimensions $d\geq 3$}, Comm. Math. Phys.
  \textbf{271} (2007), 523--559.

\bibitem{KatoPonce}
T.~Kato and G.~Ponce, \emph{Commutator estimates and the {Euler} and
  {Navier}-{Stokes} equation}, Communications on Pure and Applied Mathematics
  \textbf{41} (1988), 891--907.

\bibitem{KeelTao}
M.~Keel and T.~Tao, \emph{Endpoint {Strichartz} estimates}, American Journal of
  Mathematics \textbf{120} (1998), no.~5, 955--980.

\bibitem{KPV2}
C.~Kenig, G.~Ponce, and L.~Vega, \emph{Oscillatory integrals and regularity of
  dispersive equations}, Indiana University Mathematics Journal \textbf{40}
  (1991), no.~1, 33--69.

\bibitem{KPV3}
\bysame, \emph{The {Cauchy} problem for the {Korteweg-de Vries} equation in
  {Sobolev} spaces of negative indices}, Duke Math. J. \textbf{71} (1993),
  no.~1, 1--21.

\bibitem{KPV4}
\bysame, \emph{Small solutions to nonlinear {Schr\"odinger} equations}, Annales
  de l'{I.H.P.} Analyse non lin\'eaire \textbf{10} (1993), no.~3, 255--288.

\bibitem{KPV}
\bysame, \emph{Well-posedness and scattering results for the generalized
  {Korteweg-de Vries} equation via the contraction principle}, Communications
  on Pure and Applied Mathematics \textbf{46} (1993), no.~4, 527--620.

\bibitem{KZ}
C.~Kenig and S.~N. Ziesler, \emph{Maximal function estimates with applications
  to a modified {Kadomstev-Petviashvili} equation}, Communications on Pure and
  Applied Analysis \textbf{4} (2005), no.~1, 45--91.

\bibitem{KochSaut}
H.~Koch and J.C. Saut, \emph{Local smoothing and local solvability for third
  order dispersive equations}, SIAM Journal on Mathematical Analysis
  \textbf{38} (2007), no.~5, 1528--1541.

\bibitem{Kristin}
K.~D. Kurianski, \emph{Estimates for solutions to the {Dysthe} equation and
  numerical simulations of walking droplets in harmonic potentials}, Ph.D.
  thesis, MIT, 2019.

\bibitem{LinaresPastor}
F.~Linares and A.~Pastor, \emph{Well-posedness for the two-dimensional modified
  {Zakharov}-{Kuznetsov} equation}, SIAM Journal on Mathematical Analysis
  \textbf{41} (2009), no.~4, 1323--1339.

\bibitem{LinaresRamos}
F.~Linares and J.~Ramos, \emph{Maximal function estimates and local
  well-posedness for the generalized {Zakharov-Kuznetsov} equation}, preprint,
  available on \url{arxiv.org/abs/2005.12485}.

\bibitem{MolinetPilod}
L.~Molinet and D.~Pilod, \emph{Bilinear {Strichartz} estimates for the
  {Zakharov-Kuznetsov} equation and applications}, Annales de l'Institut Henri
  Poincare (C) Non Linear Analysis \textbf{32} (2015), no.~2, 347--371.

\bibitem{Sautnew}
R.~Mosicat, D.~Pilod, and J.-C. Saut, \emph{Global well-posedness and
  scattering for the {Dysthe} equation in ${L}^2(\mathbb{R}^2)$},  (2020),
  Private communication.

\bibitem{OnoratoMechanism2}
M.~Onorato, S.~Residori, U.~Bortolozzo, A.~Montina, and F.T. Arecchi,
  \emph{Rogue waves and their generating mechanisms in different physical
  contexts}, Physics Reports \textbf{528} (2013), no.~2, 47--89.

\bibitem{OnoratoMechanism}
S.~Randoux, P.~Walczak, M.~Onorato, and P.~Suret, \emph{Nonlinear random
  optical waves: Integrable turbulence, rogue waves and intermittency}, Physica
  D: Nonlinear Phenomena \textbf{333} (2016), 323--335.

\bibitem{RibaudVento}
F.~Ribaud and S.~Vento, \emph{A note on the {Cauchy} problem for the {2D}
  generalized {Zakharov}-{Kuznetsov} equations}, Comptes Rendus Mathematique
  \textbf{350} (2012), no.~9-10, 499--503.

\bibitem{Strich1}
M.~Ruzhansky and M.~Sugimoto, \emph{Smoothing estimates for non-dispersive
  equations}, Mathematische Annalen \textbf{365} (2016), 241--269.

\bibitem{Strich4}
A.~R. Safarov, \emph{Invariant estimates of two-dimensional oscillatory
  integrals}, Mathematical Notes \textbf{104} (2018), 293--302.

\bibitem{tao}
T.~Tao, \emph{Nonlinear dispersive equations: local and global analysis}, CBMS
  Regional Conference Series in Mathematics, vol. 106, AMS, 2006.

\end{thebibliography}
\end{document}